%% file: Redundancy_RV_arXiv.tex
\documentclass[review]{elsarticle}

\journal{Queueing Systems}

\usepackage{amsmath,amssymb,amsthm}
\usepackage{bbm}
\usepackage{todonotes}
\usepackage{enumerate}
\usepackage{verbatim}
\usepackage{hyperref}
\usepackage{lineno}
\usepackage{multirow}
\usepackage{float}

\usepackage{tikz}
\usepackage{pgfplots}
\usepgfplotslibrary{colorbrewer}
\pgfplotsset{compat=1.8}
\definecolor{clr1}{RGB}{27,158,119}
\definecolor{clr2}{RGB}{217,95,2}
\definecolor{clr3}{RGB}{117,112,179}
\definecolor{clr4}{RGB}{231,41,138}
\definecolor{clr5}{RGB}{102,166,30}
\definecolor{clr6}{RGB}{230,171,2}
\definecolor{clr7}{RGB}{166,118,29}
\pgfplotsset{
    cycle list={clr1,clr2,clr3,clr4,clr5,clr6,clr7},
}

\newtheorem{theorem}{Theorem}

\newtheorem{lemma}{Lemma}
\newtheorem{corollary}{Corollary}
\newtheorem{remark}{Remark}
\newtheorem{definition}{Definition}

\begin{document}

\begin{frontmatter}

\title{Fork-join and redundancy systems with heavy-tailed job sizes}

\author[tue]{Youri Raaijmakers\corref{mycorrespondingauthor}}
\cortext[mycorrespondingauthor]{Corresponding author}
\ead{y.raaijmakers@tue.nl}

\author[tue]{Sem Borst}

\author[tue]{Onno Boxma}

\address[tue]{Eindhoven University of Technology\\Department of Mathematics and Computer Science}

\begin{abstract}
We investigate the tail asymptotics of the response time distribution for the cancel-on-start (c.o.s.)\ and cancel-on-completion (c.o.c.)\ variants of redundancy-$d$ scheduling and the fork-join model with heavy-tailed job sizes.
We present bounds, which only differ in the pre-factor, for the tail probability of the response time in the case of the first-come first-served (FCFS) discipline. 
For the c.o.s.\ variant we restrict ourselves to redundancy-$d$ scheduling, which is a special case of the fork-join model. In particular, for regularly varying job sizes with tail index~$-\nu$ the tail index of the response time for the c.o.s.\ variant of redundancy-$d$ equals $-\min\{d_{\mathrm{cap}}(\nu-1),\nu\}$, where $d_{\mathrm{cap}} = \min\{d,N-k\}$, $N$ is the number of servers and $k$ is the integer part of the load. 
This result indicates that for $d_{\mathrm{cap}} < \frac{\nu}{\nu-1}$ the waiting time component is dominant, whereas for $d_{\mathrm{cap}} > \frac{\nu}{\nu-1}$ the job size component is dominant. 
Thus, having $d = \lceil \min\{\frac{\nu}{\nu-1},N-k\} \rceil$ replicas is sufficient to achieve the optimal asymptotic tail behavior of the response time.
For the c.o.c.\ variant of the fork-join($n_{\mathrm{F}},n_{\mathrm{J}}$) model the tail index of the response time, under some assumptions on the load, equals $1-\nu$ and $1-(n_{\mathrm{F}}+1-n_{\mathrm{J}})\nu$, for identical and i.i.d.\ replicas, respectively; here the waiting time component is always dominant.
\end{abstract}


\end{frontmatter}

\section{Introduction}
In recent years, the fork-join model has attracted strong interest. This model is a theoretical abstraction of the popular MapReduce framework~\cite{DG-MRLC}. MapReduce is a programming model for processing and generating big data sets with parallel algorithms on clusters. In MapReduce every job is divided into tasks which can be processed in parallel in any order. For completion of the job the completed tasks need to be joined together. 

\subsubsection*{Fork-join model}
In the fork-join($n_{\mathrm{F}},n_{\mathrm{J}}$) model tasks of a job are assigned to $n_{\mathrm{F}}$ servers selected uniformly at random. Redundant tasks are abandoned as soon as $n_{\mathrm{J}}$ of the $n_{\mathrm{F}}$ tasks either enter service (`cancel-on-start', c.o.s.)\ or finish service (`cancel-on-completion', c.o.c.). The job is completed when all these $n_{\mathrm{J}}$ tasks complete service.

Analytical results for the fork-join model are unfortunately scarce. Tight characterizations of the response time are only known in the special case of $n_{\mathrm{F}}=n_{\mathrm{J}}=2$, see~\cite{FH-TPQATD}. For a survey on results in other special cases we refer to~\cite{T-AFJ}. Results for the expectation of the response time are established when $n_{\mathrm{F}}=n_{\mathrm{J}} \rightarrow \infty$, see for example~\cite{BMS-FJQ,NT-AAFJ}. For a more detailed overview of the results and applications we refer to~\cite{H-OPQT}.

\subsubsection*{Redundancy scheduling}
Redundancy-$d$ scheduling is a special case within the fork-join model. 
In redundancy-$d$ scheduling replicas of a job are assigned to $d$ servers selected uniformly at random. Redundant replicas are abandoned as soon as one of the $d$ replicas either enters service or finishes service. Thus redundancy-$d$ scheduling is equivalent to the fork-join model with $n_{\mathrm{F}}=d$ and $n_{\mathrm{J}}=1$.
Observe that the c.o.s.\ variant of redundancy-$d$ is equivalent to the Join-the-Shortest-Workload-$d$ (JSW-$d$) policy, which assigns each job to the server with the smallest workload among $d$ servers selected uniformly at random, see~\cite{ABV-RCOSJSW}. The c.o.c.\ variant of redundancy-$d$ shares similarities with a strategy that assigns the job to the server that provides the minimum response time among $d$ servers selected uniformly at random, but involves possibly concurrent service of multiple replicas. 

It has been empirically shown that redundancy scheduling can improve performance in parallel-server systems~\cite{VGMSRS-LR}, especially in case of highly variable job sizes. More specifically, for large-scale applications such as Google search, the ability of redundancy scheduling to reduce the expectation and the tail of the response time has been demonstrated~\cite{DB-TAS}. 
Our understanding of redundancy scheduling is growing, and especially the stability condition for c.o.c.\ redundancy policies has received considerable attention, however, expressions for performance metrics such as the expectation or the distribution of the response time remain scarce. In~\cite{GZDHBHSW-RLR} analytical expressions for the expected response time are obtained for exponential job sizes and independent and identically distributed (i.i.d.)\ replicas. Under the assumption of asymptotic independence a fixed-point equation characterizing the response time distribution for identical and i.i.d.\ replicas is derived in~\cite{HH-PRR}. 

In this paper we examine the tail behavior of the response time when job sizes are heavy-tailed, which is one of the most relevant scenarios in redundancy scheduling and the fork-join model. Indeed, heavy tails in parallel processing are encountered in conjunction with the MapReduce framework developed at Google and its Hadoop open source implementation~\cite{DG-DPT}. Moreover, measurement studies show that workload characteristics such as file sizes, CPU times, and session lengths tend to be heavy-tailed, see~\cite{H-OPQT,PW-SSNTPE,Z-QSHT} and the references therein. The tail behavior of the waiting time distribution of the single server queue is well known, see for example~\cite{V-ABWHF} or~\cite[Chapter 2]{Z-QSHT}. Let $W_{\mathrm{FCFS}}$ denote the waiting time for the single-server queue with the FCFS discipline, for subexponential (see Definition~\ref{def: subexponential} in~\ref{app sec: preliminary results}) residual job sizes $B^{\mathrm{res}}$,
\begin{align}
\label{eq: tail behavior single server subexponential fcfs}
\mathbb{P}(W_{\mathrm{FCFS}} > x) \sim \frac{\tilde{\rho}}{1-\tilde{\rho}} \mathbb{P}(B^{\mathrm{res}}> x) ~~~ \text{as } x \rightarrow \infty,
\end{align}
where $\tilde{\rho} := \frac{\mathbb{E}[B]}{\mathbb{E}[A]}$ denotes the load with $A$ the interarrival time and $B$ the job size, and 
\begin{align*}
\mathbb{P}(B^{\mathrm{res}} > x) = \frac{1}{\mathbb{E}[B]} \int_{y=x}^{\infty} \mathbb{P}(B > y) \mathrm{d}y.
\end{align*}
In particular for regularly varying (see Definition~\ref{def: regularly varying} in~\ref{app sec: preliminary results}) job size distributions with index $-\nu$, i.e., $\mathbb{P}( B > x) =x^{-\nu} L(x)$ with $L(\cdot)$ a slowly varying function at infinity, 
\begin{align}
\label{eq: tail behavior single server fcfs}
\mathbb{P}(W_{\mathrm{FCFS}} > x) \sim \frac{\tilde{\rho}}{1-\tilde{\rho}} \frac{1}{(\nu - 1)\mathbb{E}[B]} L(x) x^{1-\nu}  ~~~ \text{as } x \rightarrow \infty.
\end{align}
One way to understand the tail index $1-\nu$ is the following. The workload (and waiting time) in an $M/G/1$ queue is distributed as a geometric($\tilde{\rho}$) sum of residual job sizes $B^{\mathrm{res}}$.
According to the theory of regular variation~\cite{BGT-RV}, loosely speaking, regular variation is preserved under integration, and asymptotically one can integrate as if $L(y)$ is kept outside the integral; so 
\begin{align}
\label{eq: tail behavior residual job sizes}
\mathbb{P}(B^{\mathrm{res}} > x) = \frac{1}{\mathbb{E}[B]} \int_{y=x}^{\infty} L(y) y^{-\nu} \mathrm{d}y \sim \frac{1}{(\nu - 1)\mathbb{E}[B]} L(x) x^{1-\nu} ~~~ \text{as } x \rightarrow \infty,
\end{align}
which implies that if $B$ is regularly varying with index $-\nu$, then $B^{\mathrm{res}}$ is regularly varying with index $1-\nu$.

The tail behavior in the single-server queue has also been studied for other service disciplines. For regularly varying job sizes, the random order of service (ROS) discipline has the same tail index as the FCFS discipline, but with a smaller pre-factor~\cite{BFLN-WTAROS}. For the last-come first-served with preemptive resume (LCFS-PR) discipline and the processor sharing (PS) discipline the tail index of the response time for regularly varying job sizes is the same as the tail index of the job size, see~\cite{Z-TABP} and~\cite{ZB-SAPS}, respectively. Thus, from a tail perspective, these service disciplines perform better than the FCFS discipline.  

Closer related to the c.o.s.\ variant of redundancy scheduling are the results for the tail behavior of the waiting time for the Join-the-Shortest-Workload (JSW) policy or equivalently the $GI/G/N$ queue, see~\cite{FK-HTMS} and~\cite{FK-LDMSHV}. The key idea in~\cite{FK-HTMS,FK-LDMSHV} to first consider deterministic interarrival times made the derivation of the tail behavior substantially more tractable.
In~\cite{FK-LDMSHV} it is shown that for long-tailed residual job sizes and $\tilde{\rho}>k$, where $k := \left\lfloor \tilde{\rho} \right\rfloor$ is the integer part of the load, 
\begin{align}
\mathbb{P}(W_{\mathrm{JSW}} > x) \geq \frac{\tilde{\rho}^{N-k}+ o(1)}{(N-k)!} \mathbb{P} \left( B^{\mathrm{res}} > \frac{\tilde{\rho}+\delta}{\tilde{\rho}-k} x \right)^{N-k} ~~~ \text{as } x\rightarrow \infty,
\label{eq: lowerbound waiting time jsw fcfs}
\end{align}
for any $\delta>0$. 
For subexponential residual job sizes and $\tilde{\rho} < k+1$ it is shown that
\begin{align}
\mathbb{P}(W_{\mathrm{JSW}} > x) \leq {N \choose k} \left(\frac{(k+1)\tilde{\rho}}{(k+1)-\tilde{\rho}}+o(1)\right)^{N-k} \mathbb{P} \left( B^{\mathrm{res}}> x (1-\delta)  \right)^{N-k} ~~~ \text{as } x\rightarrow \infty.
\label{eq: upperbound waiting time jsw fcfs}
\end{align}
A heuristic explanation for the exponent $N-k$ is as follows. 
After the arrival of $N-k$ big jobs, $N-k$ servers will be working on these big jobs for a very long time. The other $k$ servers form an unstable $GI/G/k$ system, which implies that the workload drifts linearly to infinity. Thus eventually the workload at all $N$ servers will exceed level $x$, causing the waiting time of an arriving job to be larger than $x$.

In this paper we investigate the tail behavior of the response time for both the c.o.s.\ and c.o.c.\ variants of redundancy scheduling and the fork-join model when job sizes are heavy-tailed. Throughout the paper we assume that the system under consideration is in steady state. 
For regularly varying job sizes with tail index $-\nu$ and the FCFS discipline it is shown that the response time for the c.o.s.\ variant of redundancy-$d$ has tail index $-\min\{d_{\mathrm{cap}}(\nu-1),\nu\}$, where $d_{\mathrm{cap}} = \min\{d,N-k\}$ and $k = \lfloor \tilde{\rho} \rfloor$. For small loads, this result indicates that for $d < \frac{\nu}{\nu-1}$ the waiting time component is dominant, whereas for $d > \frac{\nu}{\nu-1}$ the job size component is dominant. 
Thus, having $d = \lceil \min\{\frac{\nu}{\nu-1}\} \rceil$ replicas already achieves the optimal asymptotic tail behavior of the response time and creating even more replicas yields no improvements in terms of response time tail asymptotics. 
For high loads, the results indicate that creating many replicas yields no benefits for the tail index of the response time. 
For the c.o.c.\ variant of the more general fork-join($n_{\mathrm{F}},n_{\mathrm{J}}$) model with identical and i.i.d.\ replicas the tail index of the response time is $1-\nu$ and $1-(n_{\mathrm{F}}+1-n_{\mathrm{J}})\nu$, respectively, and the waiting time component is always dominant. Note that in this case the tail index is independent of the load of the system and for identical replicas even independent of the number of replicas. In the special case of redundancy-$d$ scheduling with identical and i.i.d.\ replicas it follows that the tail index of the response time is $1-\nu$ and $1-d\nu$, respectively. All these results for the c.o.c. variant rely on the fact that the upper bound system, which is used in the proof, is stable. The stability condition of this system does not necessarily coincide with the stability condition of the original fork-join model. 

For the LCFS-PR discipline in the fork-join model we show that the response time tail is just as heavy as the job size tail, implying that for the c.o.c.\ variant this discipline achieves better tail asymptotics than the FCFS discipline. 
For the c.o.s.\ variant the LCFS-PR discipline has better tail asymptotics than the FCFS discipline for scenarios with low load and a small number of replicas; in all other scenarios both service disciplines have similar tail asymptotics. In~\cite{RBB-STBPS} it is shown that for the c.o.c.\ variant of redundancy-$d$ scheduling with the PS discipline the tail index of the response time is $-\nu$ for identical replicas and $-d\nu$ for i.i.d.\ replicas. Table~\ref{tab: overview tail index} provides an overview of the tail index for the various models and service disciplines. 

\begin{table}[]
\small
\centering
\caption{Overview of the tail index for the c.o.s.\ and c.o.c.\ varaint of redundancy scheduling with various service disciplines where the job size is regularly varying with tail index $-\nu$. The star indicates that for this scenario we obtained results for the more general fork-join($n_{\mathrm{F}},n_{\mathrm{J}}$) model.}
\label{tab: overview tail index}
\begin{tabular}{l||l|l|l|l|}
\multirow{2}{*}{} & \multicolumn{2}{c|}{c.o.s.} &  \multicolumn{2}{c|}{c.o.c.} \\ \cline{2-5}
 & $GI/G/N$ & Red-$d$ & Red-$d$ & Red-$d$ \\
 & (Red-$N$) & & (identical) & (i.i.d.) \\ \hline \hline
FCFS & $-\min\{(N-k)(\nu-1),\nu\}$~\cite{FK-HTMS,FK-LDMSHV} & $-\min\{d_{\mathrm{cap}}(\nu-1),\nu\}$ & $1-\nu$ $(*)$ & $1-d\nu$ $(*)$ \\ \hline
LCFS-PR &  & $-\nu$ & $-\nu$ $(*)$ & $-d\nu$ $(*)$ \\ \hline
PS &  & $-\nu$ & $-\nu$~\cite{RBB-STBPS} & $-d\nu$~\cite{RBB-STBPS}
\end{tabular}
\end{table}

The remainder of the paper is organized as follows. In Section~\ref{sec: model description} we provide a model description and state preliminary results. In Section~\ref{sec: fcfs} we characterize the tail behavior of the response time for the c.o.s.\ variant of redundancy scheduling and the c.o.c.\ variant of the more general fork-join model with the FCFS discipline, with some proofs deferred to the appendix. In Section~\ref{sec: lcfs} we discuss the tail behavior in the fork-join model with the LCFS-PR discipline. Section~\ref{sec: numerical results} provides numerical results on the tail behavior of the response time in redundancy scheduling with Pareto distributed job sizes. Section~\ref{sec: conclusion} contains conclusions and some suggestions for further research. The paper ends with two appendices. \ref{app sec: preliminary results} collects various definitions and results for heavy-tailed random variables, which will be used in the paper. \ref{app sec: proof of theorem cos fcfs} provides the proof for one of the main theorems in this paper. 

\section{Model description and preliminaries}
\label{sec: model description}
Consider a system of $N$ parallel unit-speed servers. Jobs arrive at the epochs of an arbitrary renewal process, with successive interarrival times $A_{i}$, $i \geq 1$, each distributed as a generic random variable $A$.  
When a job arrives, a dispatcher assigns replicas of the job to $n_{\mathrm{F}}$ servers chosen uniformly at random (without replacement), where $1 \leq n_{\mathrm{F}} \leq N$.
We consider two possible variants where redundant replicas are abandoned as soon as $n_{\mathrm{J}}$ of the $n_{\mathrm{F}}$ replicas either enter service (c.o.s.)\ or finish service (c.o.c.). If in the c.o.s.\ variant multiple replicas enter service at exactly the same time, then one of these replicas is chosen uniformly at random and starts service. A special case of the fork-join model is redundancy-$d$ scheduling, where $n_{\mathrm{F}}=d$ and $n_{\mathrm{J}}=1$. Note that in the c.o.s.\ variant of redundancy-$d$ the dependency structure between the replicas does not play a role, since at all times there is only one replica of the job in service.  
In contrast, in the c.o.c.\ variant of redundancy-$d$, and also in the fork-join model, several replicas of the same job may be in service at the same time, and hence the dependency structure does matter.
We thus allow the replica sizes $B_{1},\dots,B_{n_{\mathrm{F}}}$ of a job to be governed by some joint distribution function $F_{B}(b_{1},\dots,b_{n_{\mathrm{F}}})$, where $B_{i}$, $i=1,\dots,n_{\mathrm{F}}$, are each distributed as some random variable $B$, but not necessarily independent. Special cases of the dependency structure are: i) perfect dependency between the variables, so-called identical replicas, where the job size is preserved for all replicas, i.e., $B_{i}=B$, for all $i=1,\dots,n_{\mathrm{F}}$, ii)  no dependency at all, so-called i.i.d.\ replicas. 

Finally, let us denote the steady-state waiting times of the replicas at their $n_{\mathrm{F}}$ servers (the time until their service starts if they are still in the system) by $W_1,\dots,W_{n_{\mathrm{F}}}$ and the steady-state response time by $R$. Let $X_{(n_{\mathrm{J}})}$ denote the $n_{\mathrm{J}}$th order statistic of a set of random variables $X_{1},\dots,X_{N}$. 

\section{FCFS discipline}
\label{sec: fcfs}
In this section we analyze the tail asymptotics of the response time with the FCFS discipline. For the c.o.s.\ variant (Section~\ref{sec: cos fcfs}) we restrict ourselves to redundancy-$d$ scheduling, whereas for the c.o.c.\ variant (Section~\ref{sec: coc fcfs}) we allow for the more general fork-join model. 

\subsection{Cancel-on-start}
\label{sec: cos fcfs}
Observe that the steady-state response time in the c.o.s.\ variant of redundancy-$d$ is given by
\begin{align}
R= \min\{W_{1},\dots,W_{d}\} +  B.
\label{eq: response time cos fcfs}
\end{align}
We refer to the time between the arrival of a job and the moment the first replica goes into service as the waiting time $W_{\mathrm{min}} = \min\{W_{1},\dots,W_{d}\}$ of a job. 
As mentioned earlier, the c.o.s.\ variant of redundancy-$d$ is equivalent to the Join-the-Shortest-Workload-$d$ (JSW-$d$) policy, which assigns each job to the server with the smallest workload among $d$ servers selected uniformly at random. 

For general interarrival times and job sizes the stability condition for the system with the JSW-$d$ policy and FCFS is given by $\tilde{\rho} = \frac{\mathbb{E}[B]}{\mathbb{E}[A]} < N$, see~\cite{FC-OSPAMS}.

In~\cite[Theorem 1.6]{FK-LDMSHV} lower and upper bounds are derived for the tail probability of the waiting time for the JSW policy. The same methodology can be used to find lower and upper bounds for JSW-$d$, and hence for the c.o.s.\ variant of redundancy scheduling with $1 \leq d \leq N$ replicas, resulting in Theorem~\ref{thm: tail waiting time cos fcfs}. The two derived lower bounds in this theorem hold for every value of $\tilde{\rho}$, but they are asymptotically dominant for different regions of $\tilde{\rho}$, as explained after the theorem. Note that for $d=N$ Theorem~\ref{thm: tail waiting time cos fcfs} recovers the results of~\cite{FK-LDMSHV} as captured in~\eqref{eq: lowerbound waiting time jsw fcfs} and~\eqref{eq: upperbound waiting time jsw fcfs}, whereas for $d=1$ the system is equivalent to a $GI/G/1$ queue for which the tail behavior is given by~\eqref{eq: tail behavior single server fcfs}.

\begin{theorem}
\label{thm: tail waiting time cos fcfs}
Consider the c.o.s.\ variant of redundancy-$d$ scheduling with the FCFS discipline.
Let $k = \lfloor \tilde{\rho} \rfloor \in \{0,1,\dots,N-1\}$ be the integer part of the load and $\delta>0$.\\ 
i) If the residual job size $B^{\mathrm{res}}$ is long-tailed, then 
\begin{align}
\label{eq: thm lowerbound waiting time cos fcfs low load}
\mathbb{P}(W_{\mathrm{min}} > x) \geq \frac{1}{{N \choose d}} \frac{\tilde{\rho}^{d}+ o(1)}{d!} \left( \bar{B}^{\mathrm{res}} \left(\left(1 + \delta \right)x \right) \right)^{d}.
\end{align}
ii) If $\tilde{\rho} < N - d$ and the residual job size $B^{\mathrm{res}}$ is subexponential, then
\begin{align}
\label{eq: thm upperbound waiting time cos fcfs low load}
\mathbb{P}(W_{\mathrm{min}} > x) \leq {N \choose d} \left(\frac{(k+1)\tilde{\rho}}{k+1-\tilde{\rho}}+o(1)\right)^{d} \left( \bar{B}^{\mathrm{res}} \left( \frac{x (1-\delta)}{k+1} \right) \right)^{d}.
\end{align}
iii) If the residual job size $B^{\mathrm{res}}$ is long-tailed, then 
\begin{align}
\label{eq: thm lowerbound waiting time cos fcfs high load}
\mathbb{P}(W_{\mathrm{min}} > x) \geq \frac{\tilde{\rho}^{N-k}+ o(1)}{(N-k)!} \left( \bar{B}^{\mathrm{res}}\left(\frac{\tilde{\rho}+\delta}{\tilde{\rho}- k} x\right) \right)^{N-k}.
\end{align}
iv) If $\tilde{\rho} > N - d$ and the residual job size $B^{\mathrm{res}}$ is subexponential, then
\begin{align}
\label{eq: thm upperbound waiting time cos fcfs high load}
\mathbb{P}(W_{\mathrm{min}} > x) \leq {N \choose k} \left(\frac{(k+1)\tilde{\rho}}{k+1-\tilde{\rho}}+o(1)\right)^{N-k} \left( \bar{B}^{\mathrm{res}} \left( \frac{(k+1-N+d)x (1-\delta)}{k+1} \right) \right)^{N-k}.
\end{align}
\end{theorem}
\begin{proof}
Let $\boldsymbol{V}=(V_{1},\dots,V_{N})$ denote the vector of residual workloads of the servers. Recall that $V_{(i)}$ denotes the $i$th order statistic of the set $V_{1},\dots,V_{N}$.
The proof of i) follows from the inequality
\begin{align*}
\mathbb{P}(W_{\mathrm{min}} > x) \geq \frac{1}{{N \choose d}} \mathbb{P}(V_{(1)} > x,\dots,V_{(d)} > x),
\end{align*}
with $\frac{1}{{N \choose d}}$ corresponding to the probability that the replicas of an arbitrary job are assigned to the servers with the $d$ largest workloads, and where 
\begin{align*}
\mathbb{P}(V_{(1)} > x,\dots,V_{(d)} > x) \geq \frac{\tilde{\rho}^{d}+ o(1)}{d!} \left( \bar{B}^{\mathrm{res}} \left( \left(1 + \delta \right) x \right) \right)^{d},
\end{align*}
by similar arguments as in the proof of Lemma~3.1 in~\cite{FK-LDMSHV}. 
The proof of iii) follows from the inequality
\begin{align*}
\mathbb{P}(W_{\mathrm{min}} > x) \geq \mathbb{P}(V_{1} > x,\dots,V_{N} > x),
\end{align*}
where 
\begin{align*}
\mathbb{P}(V_{1} > x,\dots,V_{N} > x) \geq \frac{\tilde{\rho}^{N-k}+ o(1)}{(N-k)!} \left( \bar{B}^{\mathrm{res}}\left(\frac{\tilde{\rho}+\delta}{\tilde{\rho}- k} x\right) \right)^{N-k},
\end{align*}
by similar arguments as in the proof of Theorem~5.1 in~\cite{FK-LDMSHV}.
The proof of ii) and iv) can be found in~\ref{app sec: proof of theorem cos fcfs}.
\end{proof}

As reflected in the proof sketches, the asymptotic lower bounds in~\eqref{eq: thm lowerbound waiting time cos fcfs low load} and~\eqref{eq: thm lowerbound waiting time cos fcfs high load} correspond to two different scenarios for a large value of $W_{\mathrm{min}}$ to occur. \\
Scenario~$1$ involves the arrival of $d$ jobs of size $x$ or larger `overlapping in time'.
In the JSW-$d$ system these jobs will be assigned to $d$ different servers with overwhelming probability for large $x$, and thus the workload at these $d$ servers will exceed $x$.
A newly arriving job that is so unfortunate as to sample exactly these $d$ servers (which happens with probability  $1/{N \choose d}$) will experience a waiting time larger than $x$.
Scenario~$2$ involves the arrival of $N-k$ sufficiently large jobs `overlapping in time', which instantaneously causes the workloads at $N-k$ servers to become large as described above, assuming $N - k \leq d$.
This will also result in subsequent jobs all being assigned to one of the other $k$~servers and hence create overload, so that the workloads at these servers will gradually start growing.
Thus, eventually the workloads at all servers will be large, and every arriving job will experience a large waiting time.
Observe that this scenario corresponds to that in the GI/G/N queue discussed in~\cite{FK-LDMSHV}, as illustrated by the match with Equation~\eqref{eq: lowerbound waiting time jsw fcfs}.

Scenarios~$1$ and~$2$ are asymptotically dominant in case $d \leq N-k$ and $d \geq N-k$, respectively, reflecting that a large waiting time is most likely due to a minimum number of $d_{\mbox{cap}} = \min\{d, N - k\}$ large jobs.
Note that in Scenario~$1$ the workloads at all servers will in fact grow large as well when $d \geq N - k$, but that Scenario~$2$ dominates in that case.

Scenarios with large workloads at $l$~servers, with $d < l < N$, do not asymptotically contribute to the probability of a large waiting time.
This may be intuitively explained by observing the following.
(i) If such scenarios involve strictly more than $d$ large workloads without resulting in overload of all servers (so $d < l < N-k$) then they are asymptotically much less likely
than Scenario 1.
(ii) If such scenarios involve $l \geq N-k$ large workloads, this will quickly result in overload of all servers, just like in Scenario 2.

\begin{corollary}[Analogous to Corollary 1.1 in~\cite{FK-LDMSHV}]
\label{col: tail waiting time dominated varying and long-tailed cos fcfs}
Let the residual job size $B^{\mathrm{res}}$ be long-tailed and dominated varying and $k<\tilde{\rho}<k+1$, i.e., $\tilde{\rho}$ not an integer value. Then there exist constants $c_{1}$ and $c_{2}$ such that, for all $x$,
\begin{align*}
c_{1} \left( \bar{B}^{\mathrm{res}}(x ) \right)^{d_{\mathrm{cap}}}  \leq \mathbb{P}(W_{\mathrm{min}} > x) \leq c_{2} \left( \bar{B}^{\mathrm{res}}(x ) \right)^{d_{\mathrm{cap}}},
\end{align*}
where $d_{\mathrm{cap}} = \min\{d,N-k\}$.
\end{corollary}
\begin{proof}
The result follows directly from Theorem~\ref{thm: tail waiting time cos fcfs}, the last inclusion in~\eqref{eq: relation tail classes} and the definition of dominated variation (Definition~\ref{def: dominated varying} in~\ref{app sec: preliminary results}). 
\end{proof}

\begin{remark}
Note that in Corollary~\ref{col: tail waiting time dominated varying and long-tailed cos fcfs} we exclude integer values for the load. Most of the heavy-tail results focus on the case where the load is not an integer, since the integer case is significantly more delicate to analyze. For a detailed study on the integer case in the $GI/G/2$ queueing system we refer to~\cite{BM-TADHL}.
\end{remark}

\begin{corollary}
\label{cor: tail waiting response time regularly varying cos fcfs}
For the c.o.s.\ variant of redundancy-$d$ scheduling with the FCFS discipline:
\begin{enumerate}[i)]
\item if $B \in RV(-\nu)$, then $W_{\mathrm{min}} \in ORV(d_{\mathrm{cap}}(1-\nu))$,
\item if $B \in RV(-\nu)$, then $R \in ORV(-\min\{d_{\mathrm{cap}}(\nu-1),\nu\})$. 
\end{enumerate}
\end{corollary}
\begin{proof}
It is well known that if $B \in RV(-\nu)$, then $B^{\mathrm{res}} \in RV(1-\nu)$, see~\eqref{eq: tail behavior residual job sizes}. The proof of i) follows by applying this result to Corollary~\ref{col: tail waiting time dominated varying and long-tailed cos fcfs} together with the inclusion $RV \subset \mathcal{L} \cap \mathcal{D}$ from~\eqref{eq: relation tail classes} and Lemma~\ref{lem: bounds in o regularly varying} in~\ref{app sec: preliminary results}. 
The proof of ii) follows by i), Equation~\eqref{eq: response time cos fcfs} and Lemma~\ref{lem: sum o regularly varying} in~\ref{app sec: preliminary results}. 
\end{proof}

From Corollary~\ref{cor: tail waiting response time regularly varying cos fcfs} we conclude that the waiting time component is dominant in the response time tail as long as $d_{\mathrm{cap}} \leq \frac{\nu}{\nu-1}$, but otherwise the job size component is dominant. Better than that ($x^{-\nu}$ tail behavior) is, obviously, not possible for the response time. In other words, having more than $\frac{\nu}{\nu-1}$ replicas will not provide any improvement in the tail behavior. For example, consider a system with a sufficiently small load. If $\nu=4/3$, then $d=4$ already yields $R \in ORV(-\nu)$, and from a tail perspective choosing $d>4$ yields no benefits. If $\nu=3/2$, then it does not pay to take $d$ larger than $3$. If $\nu \geq 2$ (so $B$ has a finite second moment), then it does not pay to take $d$ larger than $2$. For high loads, the results indicate that creating many replicas yields no benefits for the tail index of the response time. 

\subsection{Cancel-on-completion}
\label{sec: coc fcfs}
In this section we analyze the tail behavior for the c.o.c.\ variant of the fork-join($n_{\mathrm{F}},n_{\mathrm{J}}$) model. The steady-state response time is given by
\begin{align}
R= (W+B)_{(n_{\mathrm{J}})}.
\label{eq: response time coc fcfs}
\end{align}

Next, we provide an upper and lower bound for the waiting time and response time via the workload in an auxiliary single-server queue, which is similar to~\cite[Lemma 1]{RBB-RSSB}.
\begin{lemma}
\label{lem: upperbound waiting time fcfs}
The waiting time $W_{(n_{\mathrm{J}})}$ and the response time $R$ in the c.o.c.\ variant of the fork-join($n_{\mathrm{F}},n_{\mathrm{J}}$) model with the FCFS discipline are stochastically bounded from above by the waiting time $W_{\mathrm{U}}$ and response time $R_{\mathrm{U}}$, respectively, in a $GI/G/1/FCFS$ queue with interarrival time $A$ and job size $B_{(n_{\mathrm{J}})}$.
\end{lemma}
\begin{proof}
Consider an auxiliary system in which all jobs are assigned to the same $n_{\mathrm{F}}$ servers and where servers wait with serving a new job until all the $n_{\mathrm{J}}$ replicas are finished. This system is equivalent to the $GI/G/1/FCFS$ queue as defined in the lemma. Let $\omega_{i}$ be the workload at server $i$, where we define the workload as the amount of work a server needs to complete to become idle in the absence of any arrivals. By induction it can be shown that $\omega_{i}$ is bounded from above by the workload $\omega_{\mathrm{U}}$ in the auxiliary $GI/G/1/FCFS$ system at all times. Assume that $\omega_{(N)} \leq \omega_{\mathrm{U}}$ after the $m$-th arrival. Then, after the $(m+1)$-th arrival the new workload is 
\begin{align*}
\omega_{s^{m+1}_{l}} = \max\{(\omega_{\boldsymbol{s}^{m+1}} + b)_{(n_{\mathrm{J}})},\omega_{s_{l}}\} \leq \max\{ (\omega_{(N)} + b)_{(n_{\mathrm{J}})},\omega_{(N)}\} = \omega_{(N)} + b_{(n_{\mathrm{J}})},
\end{align*}
for $l=1,\dots,n$, since $\omega_{i} \leq \omega_{(N)}$ for all $i=1,\dots,N$.
Thus the increase in maximum workload is bounded by $b_{(n_{\mathrm{J}})}$, which is exactly the increase in workload in the corresponding $GI/G/1/FCFS$ queue.
Observe that the bound for the workload implies $W_{i} \leq W_{\mathrm{U}}$ for all $i=1,\dots,n_{\mathrm{F}}$, from which it follows that $W_{(n_{\mathrm{J}})} \leq W_{\mathrm{U}}$ and
\begin{align*}
R= (W+B)_{(n_{\mathrm{J}})} \leq W_{\mathrm{U}} + B_{(n_{\mathrm{J}})} = R_{\mathrm{U}}.
\end{align*}
\end{proof}

\begin{lemma}
\label{lem: lowerbound waiting time fcfs}
The waiting time $W_{(n_{\mathrm{J}})}$ and the response time $R$ in the c.o.c.\ variant of the fork-join($n_{\mathrm{F}},n_{\mathrm{J}}$) model with the FCFS discipline are stochastically bounded from below by the waiting time $W_{\mathrm{L}}$ and response time $R_{\mathrm{L}}$, respectively, in a $GI/G/1/FCFS$ queue with a random selection of the arrivals based on Bernoulli experiments with probability $1/K$, i.e., mean interarrival time $K \mathbb{E}[A]$, and with a job size $B_{(n_{\mathrm{J}})}$, where $K= {N \choose n_{\mathrm{F}}} \frac{n_{\mathrm{F}}!}{(n_{\mathrm{F}}-n_{\mathrm{J}}+1)!}$.
\end{lemma}
\begin{proof}
Consider an auxiliary system in which we only allow arrivals to $n_{\mathrm{F}}$ specific servers, say $1$ up and until $n_{\mathrm{F}}$, and for which the first $n_{\mathrm{J}}$ job sizes are the smallest and in increasing order. Thus, the smallest job is assigned to server $1$, the second smallest job to server $2$, etcetera up and until server $n_{\mathrm{J}}$. Note that the other arrivals are only deleted in the auxiliary system and not in the original system, thus the amount of work cannot be lower. The auxiliary system is equivalent to the $GI/G/1/FCFS$ queue as defined in the lemma, since server $n_{\mathrm{J}}$ always finishes the $n_{\mathrm{J}}$-th replica last. Similarly to Lemma~\ref{lem: upperbound waiting time fcfs}, it can be shown that the workload $\omega_{i}$ at server $i$ in the original system, is bounded from below by the workload $\omega_{\mathrm{L}}$ in the auxiliary $GI/G/1/FCFS$ system. Observe that the bound for the workload implies $W_{i} \geq W_{\mathrm{L}}$ for all $i=1,\dots,n_{\mathrm{F}}$, from which it follows that $W_{(n_{\mathrm{J}})} \geq W_{\mathrm{L}}$ and
\begin{align*}
R= (W+B)_{(n_{\mathrm{J}})} \geq W_{\mathrm{L}} + B_{(n_{\mathrm{J}})} = R_{\mathrm{L}}.
\end{align*}
\end{proof}

A sufficient stability condition for general interarrival times and job sizes is $\rho_{\mathrm{U}} := \frac{\mathbb{E}[B_{(n_{\mathrm{J}})}]}{\mathbb{E}[A]} < 1$, which can be proved via the upper bound system given in Lemma~\ref{lem: upperbound waiting time fcfs}. The exact stability condition for the c.o.c.\ variant of the fork-join($n_{\mathrm{F}},n_{\mathrm{J}}$) model, and also redundancy-$d$ scheduling, with the FCFS discipline in such a general setting is still unknown. 

\begin{theorem}
\label{thm: tail waiting time subexponential fcfs}
If $\rho_{\mathrm{U}} < 1$ and the residual job size $B_{(n_{\mathrm{J}})}^{\mathrm{res}}$ is subexponential, then for the c.o.c.\ variant of the fork-join($n_{\mathrm{F}},n_{\mathrm{J}}$) model scheduling with the FCFS discipline:
\begin{align*}
\frac{\rho_{\mathrm{L}}}{1-\rho_{\mathrm{L}}} \bar{B}_{(n_{\mathrm{J}})}^{\mathrm{res}}(x) \leq \mathbb{P}(W_{(n_{\mathrm{J}})} > x) \leq \frac{\rho_{\mathrm{U}}}{1-\rho_{\mathrm{U}}} \bar{B}_{(n_{\mathrm{J}})}^{\mathrm{res}}(x) ~~~ \text{as } x \rightarrow \infty,
\end{align*}
where $\rho_{\mathrm{L}} =\frac{\mathbb{E}[B_{(n_{\mathrm{J}})}]}{K \mathbb{E}[A]}$ with $K={N \choose d}\frac{n_{\mathrm{F}}!}{(n_{\mathrm{F}}-n_{\mathrm{J}}+1)!}$ and $\rho_{\mathrm{U}} =\frac{\mathbb{E}[B_{(n_{\mathrm{J}})}]}{\mathbb{E}[A]}$.
\end{theorem}
\begin{proof}
\textit{Upper bound:}
By Lemma~\ref{lem: upperbound waiting time fcfs} the waiting time of a job is bounded from above by the waiting time $W_{\mathrm{U}}$ in a $GI/G/1/FCFS$ queue with interarrival time $A$ and job size $B_{(n_{\mathrm{J}})}$.
Thus, by the subexponentiality of $B_{(n_{\mathrm{J}})}^{\mathrm{res}}$, we can apply known results for the single-server queue, see~\eqref{eq: tail behavior single server subexponential fcfs}, and obtain
\begin{align}
\label{eq: upperbound waiting time probability}
\mathbb{P}(W_{\mathrm{U}} > x) \sim \frac{\rho_{\mathrm{U}}}{1-\rho_{\mathrm{U}}} \bar{B}_{(n_{\mathrm{J}})}^{\mathrm{res}}(x) ~~~ \text{as } x \rightarrow \infty.
\end{align}
\textit{Lower bound:}
By Lemma~\ref{lem: lowerbound waiting time fcfs} the waiting time of a job is bounded from below by the waiting time $W_{\mathrm{L}}$ in a $GI/G/1/FCFS$ queue with a random selection of the arrivals based on Bernoulli experiments with probability $1/K$, i.e., mean interarrival time $K \mathbb{E}[A]$, and job size $B_{(n_{\mathrm{J}})}$. Again, by the subexponentiality of $B_{(n_{\mathrm{J}})}^{\mathrm{res}}$, by applying known results for the single-server queue we obtain
\begin{align}
\label{eq: lowerbound waiting time probability}
\mathbb{P}(W_{\mathrm{L}} > x) \sim \frac{\rho_{\mathrm{L}}}{1-\rho_{\mathrm{L}}} \bar{B}_{(n_{\mathrm{J}})}^{\mathrm{res}}(x) ~~~ \text{as } x \rightarrow \infty.
\end{align}
By combining~\eqref{eq: upperbound waiting time probability} and~\eqref{eq: lowerbound waiting time probability} we get the desired statement. 
\end{proof}

The next corollary provides insight in the tail behavior when the distribution of the $n_{\mathrm{J}}$th order statistic of the job size is regularly varying, i.e., $B_{(n_{\mathrm{J}})} \in RV(-\tilde{\nu})$. Observe that, in the special case of identical replicas $B_{(n_{\mathrm{J}})} \in RV(-\nu)$ when $B \in RV(-\nu)$, thus in this case $\tilde{\nu} = \nu$, whereas for i.i.d.\ replicas $B_{(n_{\mathrm{J}})} \in RV(-(n_{\mathrm{F}}+1-n_{\mathrm{J}})\nu)$ when $B \in RV(-\nu)$ (see~\cite{JM-RVF}), thus in this case $\tilde{\nu} = (n_{\mathrm{F}}+1-n_{\mathrm{J}})\nu$.

\begin{corollary}
\label{cor: tail waiting time regularly varying fcfs}
For the c.o.c.\ variant of the fork-join($n_{\mathrm{F}},n_{\mathrm{J}}$) model with the FCFS discipline and $\rho_{\mathrm{U}} < 1$:
\begin{enumerate}[i)]
\item if $B_{(n_{\mathrm{J}})} \in RV(-\tilde{\nu})$, then $W_{(n_{\mathrm{J}})} \in ORV(1-\tilde{\nu})$,
\item if $B_{(n_{\mathrm{J}})} \in RV(-\tilde{\nu})$, then $R \in ORV(1-\tilde{\nu})$.
\end{enumerate}
\end{corollary}
\begin{proof}
For regularly varying residual job sizes we know that
\begin{align*}
\mathbb{P}(B_{(n_{\mathrm{J}})}^{\mathrm{res}} > x) \sim \frac{1}{(\tilde{\nu} - 1)\mathbb{E}[B_{(n_{\mathrm{J}})}]} L(x) x^{1-\tilde{\nu}} ~~~ \text{as } x \rightarrow \infty,
\end{align*}
see~\eqref{eq: tail behavior residual job sizes}.
The proof of i) follows by Theorem~\ref{thm: tail waiting time subexponential fcfs} and Lemma~\ref{lem: bounds in o regularly varying}.
For the response time we can again use Lemmas~\ref{lem: upperbound waiting time fcfs} and~\ref{lem: lowerbound waiting time fcfs} as in Theorem~\ref{thm: tail waiting time subexponential fcfs}. Using the known result for the tail behavior in the single-server queue, see~\eqref{eq: tail behavior single server fcfs}, together with Lemma~\ref{lem: sum minimum regularly varying} we obtain that
\begin{align*}
\mathbb{P}(R > x) \geq \mathbb{P}(R_{\mathrm{L}} > x) = \mathbb{P}(W_{\mathrm{L}} + B_{(n_{\mathrm{J}})} > x) \sim \frac{\rho_{\mathrm{L}}}{1-\rho_{\mathrm{L}}} \frac{L(x) x^{1-\tilde{\nu}}}{(\tilde{\nu} - 1)\mathbb{E}[B_{(n_{\mathrm{J}})}]}~~~ \text{as } x \rightarrow \infty,
\end{align*}
and
\begin{align*}
\mathbb{P}(R > x) \leq \mathbb{P}(R_{\mathrm{U}} > x) = \mathbb{P}(W_{\mathrm{U}} + B_{(n_{\mathrm{J}})} > x) \sim \frac{\rho_{\mathrm{U}}}{1-\rho_{\mathrm{U}}} \frac{L(x) x^{1-\tilde{\nu}}}{(\tilde{\nu} - 1)\mathbb{E}[B_{(n_{\mathrm{J}})}]} ~~~ \text{as } x \rightarrow \infty.
\end{align*}
Now we can apply Lemma~\ref{lem: bounds in o regularly varying} in~\ref{app sec: preliminary results} and obtain the desired result. 
\end{proof}

\begin{remark}
For identical replicas we can even find a better upper bound in Theorem~\ref{thm: tail waiting time subexponential fcfs}. Indeed, consider the system in which all replicas are completely served. This system is equivalent to a $GI/G/1/FCFS$ queue with a random selection of the arrivals based on Bernoulli experiments with probability $n_{\mathrm{F}}/N$, i.e., mean interarrival time $\frac{N \mathbb{E}[A]}{n_{\mathrm{F}}}$, and job size $B$, which is equal to $B_{(n_{\mathrm{J}})}$ in the case of identical replicas. 
\end{remark}

Observe that all the results for the tail index rely on the fact that the upper bound system is stable. The stability condition of this system does not necessarily coincide with the stability condition of the original fork-join model. We conjecture that these tail index results are valid whenever the original fork-join model is stable. However, note that constructing a tractable upper bound system with the same stability condition as the original fork-join model is hard, because this stability condition is unknown. 

Interestingly, in contrast to the c.o.s.\ variant of redundancy, we observe that the tail index in the c.o.c.\ variant of the fork-join model does not depend on the load of the system. The main difference between the two variants is that for the c.o.s.\ variant we need multiple big jobs for a large value of $W_{\mathrm{min}}$ to occur, whereas for the c.o.c.\ variant we only need one big job. Moreover, note that a big job means that at least $n_{\mathrm{F}}+1-n_{\mathrm{J}}$ replica sizes should be big since we cancel the redundant replicas as soon as the first $n_{\mathrm{J}}$ replicas complete service. This is the reason why for i.i.d.\ replicas we get the tail index $1-(n_{\mathrm{F}}+1-n_{\mathrm{J}})\nu$ and for identical replicas $1-\nu$.

In the remainder of this subsection we focus on two special cases of the dependency structure, namely identical and i.i.d.\ replicas. 

For the special case of identical replicas in the c.o.c.\ variant of the fork-join model with the FCFS discipline we have concluded:
if $B \in RV(-\nu)$, then $R \in ORV(1-\nu)$ which is independent of the number of replicas.
We may conclude that the tail index is the same as for the single-server queue, see~\eqref{eq: tail behavior single server fcfs}. Moreover, if we compare the tail index of the c.o.s.\ and c.o.c.\ variants of redundancy scheduling with identical replicas it follows that the c.o.s.\ variant always performs better from a tail perspective. 

For the special case of i.i.d.\ replicas in the c.o.c.\ variant of the fork-join model with the FCFS discipline we have concluded:
if $B \in RV(-\nu)$, then $R \in ORV(1-(n_{\mathrm{F}}+1-n_{\mathrm{J}})\nu)$. 
If $n_{\mathrm{F}}=n_{\mathrm{J}}=1$, then $R \in ORV(1-\nu)$ which is consistent with the case of identical replicas. Moreover, if we compare the tail index of the c.o.s.\ and c.o.c.\ variants of redundancy scheduling with i.i.d.\ replicas it follows that the c.o.c.\ variant always performs better from a tail perspective. Observe that this statement is in contrast with the case for identical replicas.  

We studied two special structures for the dependency between replicas. The general case, with a vector $(B_1,\dots,B_{n_{\mathrm{F}}})$ of possibly dependent and multivariate regularly varying job sizes, will be more involved.
For further information on multivariate regular variation we refer to~\cite{R-HTA} or~\cite[Appendix A1.5]{BGT-RV} and the references therein. 

We determined the tail behavior for the c.o.s.\ variant of redundancy scheduling and the c.o.c.\ variant of the more general fork-join model. It can be concluded that the analysis of the c.o.s.\ variant is much more challenging than of the c.o.c.\ variant. One of the reasons is that for the c.o.s.\ variant multiple big jobs might be needed to have a large waiting time while for the c.o.c.\ variant only one big job is needed. In some sense this is remarkable, since for the stability condition it is the other way around: The stability condition for the c.o.s.\ variant of redundancy scheduling is known, whereas for the c.o.c.\ variant of the fork-join model, and also redundancy-$d$ scheduling, it is still an open problem for non-exponential job size distributions.

\section{LCFS-PR discipline}
\label{sec: lcfs}
In this section we study the tail behavior of the response time in the fork-join model with the LCFS-PR discipline. First, we discuss known results for the single-server queue and in Sections~\ref{sec: cos lcfs} and~\ref{sec: coc lcfs} the tail behavior for the c.o.s.\ and c.o.c.\ variants of the fork-join model is discussed, respectively. 

For the $GI/G/1$ queue with regularly varying job sizes the tail behavior of the response time distribution is known
\begin{align}
\label{eq: tail behavior single server lcfs}
\mathbb{P}(R_{\text{LCFS-PR}} > x) \sim \mathbb{E}[N_{\text{bp}}] (1-\tilde{\rho})^{-\nu} L(x) x^{-\nu} ~~~ \text{as } x \rightarrow \infty,
\end{align}
where $N_{\text{bp}}$ denotes the number of jobs completed during a busy period, see~\cite{Z-TABP}. 
One way to understand~\eqref{eq: tail behavior single server lcfs} is the following. 
First observe that for the LCFS-PR discipline
\begin{align*}
R_{\text{LCFS-PR}} \stackrel{d}{=} P,
\end{align*}
where $P$ is the busy period of a $GI/G/1$ queue. 
Let $V(t)$ be the amount of work in the system at time $t$ and assume that the first job arrives in an empty system at time $0$. The busy period $P$ is then defined as
\begin{align*}
P := \inf\{t>0 : V(t) = 0\}. 
\end{align*}
Let the cycle maximum $C_{\text{max}}$ be defined by
\begin{align*}
C_{\text{max}} := \sup\{V(t), 0 \leq t \leq P\}.
\end{align*}
It is shown, see for example~\cite[Corollary 2.2]{HRS-PBOCQ}, that subexponentiality of $B$ implies that $\mathbb{P}(C_{\text{max}} > x) \sim \mathbb{P}(W_{\text{max}} > x)$, where $W_{\text{max}}$ is the maximum waiting time during a busy period, and from~\cite{A-SEASP} we know that,
\begin{align*}
\mathbb{P}(W_{\text{max}} > x) \sim \mathbb{E}[N_{\mathrm{bp}}] \mathbb{P}(B > x) ~~~ \text{as } x \rightarrow \infty.
\end{align*}
Combining both relations gives
\begin{align*}
\mathbb{P}(C_{\text{max}} > x) \sim \mathbb{E}[N_{\mathrm{bp}}] \mathbb{P}(B > x) ~~~ \text{as } x \rightarrow \infty.
\end{align*}
A large maximum waiting time is most likely due to one large job. After this large job, the system behaves normally and the workload goes to zero with negative drift $-(1-\tilde{\rho})$. Hence if $C_{\text{max}}$ is large, then one would expect that
\begin{align*}
P \approx \frac{C_{\text{max}}}{1-\tilde{\rho}},
\end{align*}
from which it follows that
\begin{align*}
\mathbb{P}(P > x) \sim \mathbb{E}[N_{\mathrm{bp}}] (1-\tilde{\rho})^{-\nu} L(x) x^{-\nu} ~~~ \text{as } x \rightarrow \infty.
\end{align*}
Observing that the busy period coincides with the response time of a job for the LCFS-PR discipline gives the desired result in~\eqref{eq: tail behavior single server lcfs}. 

\subsection{Cancel-on-start}
\label{sec: cos lcfs}
Note that for the LCFS-PR discipline the c.o.s.\ variant of the fork-join($n_{\mathrm{F}},n_{\mathrm{J}}$) model is equivalent to the system where replicas of each job are assigned to $n_{\mathrm{J}}$ servers chosen uniformly at random (without replacement), since all replicas immediately go into service.  
Thus, each queue is equivalent with a $GI/G/1/LCFS \text{-} PR$ queue with a random selection of the arrivals based on Bernoulli experiments with probability $n_{\mathrm{J}}/N$, i.e., mean interarrival time $N \mathbb{E}[A]/n_{\mathrm{J}}$ and mean job size $\mathbb{E}[B]$. Hence the stability condition is $\tilde{\rho} < \frac{N}{n_{\mathrm{J}}}$. For regularly varying job sizes the tail behavior of the response time is given by
\begin{align*}
\mathbb{P}(R > x) = \mathbb{P}(\max_{i=1,\dots,n_{\mathrm{J}}}R_{\mathrm{LCFS-PR}} > x) \sim n_{\mathrm{J}} \mathbb{P}(R_{\mathrm{LCFS-PR}} > x) ~~~ \text{as } x \rightarrow \infty,
\end{align*}
see for example~\cite{M-RVSAP}. 

Observe that a similar reasoning is applicable for any service discipline in which all replicas immediately go into service. Another example is the processor-sharing (PS) discipline for which the tail behavior of the response time for the single-server queue with regularly varying job sizes with index $-\nu$ is given by
\begin{align*}
\mathbb{P}(R_{\text{PS}} > x) \sim (1-\tilde{\rho})^{-\nu} L(x) x^{-\nu} ~~~ \text{as } x \rightarrow \infty,
\end{align*} 
see for example~\cite[Chapter 3]{Z-QSHT} or~\cite{ZB-SAPS}. 

\subsection{Cancel-on-completion}
\label{sec: coc lcfs}
In this section we analyze the tail asymptotics for the c.o.c.\ variant of the fork-join($n_{\mathrm{F}},n_{\mathrm{J}}$) model with the LCFS-PR discipline. 
Similarly to the FCFS discipline in Section~\ref{sec: coc fcfs} we first state a lemma that provides an upper bound for the response time. 

\begin{lemma}
\label{lem: upperbound response time lcfs}
The response time in the fork-join($n_{\mathrm{F}},n_{\mathrm{J}}$) model with the LCFS-PR discipline is stochastically bounded from above by the response time in a $GI/G/1/LCFS \text{-} PR$ queue with interarrival time $A$ and job size $B_{(n_{\mathrm{J}})}$.
\end{lemma}
\begin{proof}
Consider an auxiliary system in which all jobs are assigned to the same $n_{\mathrm{F}}$ servers and where servers wait with serving a new job until all the $n_{\mathrm{J}}$ replicas are finished. This system is equivalent to the $GI/G/1/LCFS \text{-} PR$ queue as defined in the lemma. Let $\omega_{i}$ be the workload at server $i$, where we define the workload as the amount of work a server needs to complete to become idle in the absence of any arrivals. It can be shown, by taking similar steps as in Lemma~\ref{lem: upperbound waiting time fcfs}, that $\omega_{i}$ 
is bounded from above by the workload $\omega_{\mathrm{U}}$ in the auxiliary $GI/G/1/LCFS \text{-} PR$ system at all times. Observe that the bound for the workload implies $R \leq R_{\mathrm{U}}$.
\end{proof}

A sufficient stability condition for general interarrival times and job sizes is $\rho_{\mathrm{U}} = \frac{\mathbb{E}[B_{(n_{\mathrm{J}})}]}{\mathbb{E}[A]} < 1$, which can be proved via the upper bound system given in Lemma~\ref{lem: upperbound response time lcfs}. The exact stability condition of the c.o.c.\ variant of the fork-join($n_{\mathrm{F}},n_{\mathrm{J}}$) model, and also redundancy-$d$ scheduling, with the LCFS-PR discipline in such a general setting is still unknown. 

The next theorem provides insight in the tail behavior when the distribution of the $n_{\mathrm{J}}$th order statistic of the job size is regularly varying, i.e., $B_{(n_{\mathrm{J}})} \in RV(-\tilde{\nu})$. Similarly to the FCFS discipline it includes the special cases of identical replicas ($\tilde{\nu} = \nu$) and i.i.d.\ replicas ($\tilde{\nu} = (n_{\mathrm{F}}+1-n_{\mathrm{J}})\nu$).

\begin{theorem}
\label{thm: tail response time regularly varying lcfs}
For the c.o.c.\ variant of the fork-join($n_{\mathrm{F}},n_{\mathrm{J}}$) model with the LCFS-PR discipline and $\rho_{\mathrm{U}} < 1$:
if $B_{(n_{\mathrm{J}})} \in RV(-\tilde{\nu})$, then $R \in ORV(-\tilde{\nu})$.
\end{theorem}
\begin{proof}
\textit{Upper bound:}
By Lemma~\ref{lem: upperbound response time lcfs} the response time is bounded from above by the response time in a $GI/G/1/LCFS \text{-} PR$ queue with interarrival time $A$ and job size $B_{(n_{\mathrm{J}})}$. Let $R_{\mathrm{U}}$ denote the response time in this upper bound system. 
Since $B_{(n_{\mathrm{J}})}$ is regularly varying, we can apply known results for the single-server queue, see~\eqref{eq: tail behavior single server lcfs}, and obtain
\begin{align*}
\mathbb{P}(R_{\mathrm{U}} > x) \sim \mathbb{E}[N_{\mathrm{bp}}] (1-\rho_{\mathrm{U}})^{-\tilde{\nu}} L(x) x^{-\tilde{\nu}} ~~~ \text{as } x \rightarrow \infty,
\end{align*}
where $\rho_{\mathrm{U}}=\frac{\mathbb{E}[B_{(n_{\mathrm{J}})}]}{\mathbb{E}[A]}$.\\
\textit{Lower bound:} One could argue that $R$ cannot have a heavier tail than $R_{\mathrm{U}}$, but also not a lighter tail, since
\begin{align*}
\mathbb{P}(R > x) \geq \mathbb{P}(B_{(n_{\mathrm{J}})} > x) = L(x) x^{-\tilde{\nu}}, ~~~x>0.
\end{align*}
The proof follows by Lemma~\ref{lem: bounds in o regularly varying} in~\ref{app sec: preliminary results}.
\end{proof}

\begin{remark}
For identical replicas we can even find a better upper bound in Theorem~\ref{thm: tail response time regularly varying lcfs}. Indeed, consider the system in which all replicas are completely served. This system is equivalent to a $GI/G/1/LCFS \text{-} PR$ queue with a random selection of the arrivals based on Bernoulli experiments with probability $n_{\mathrm{F}}/N$, i.e., mean interarrival time $\frac{N \mathbb{E}[A]}{n_{\mathrm{F}}}$, and job size $B$, which is equal to $B_{(n_{\mathrm{J}})}$ in the case of identical replicas. 
\end{remark}

Theorem~\ref{thm: tail response time regularly varying lcfs} indicates that for the LCFS-PR discipline the tail of the response time is just as heavy as the tail of the job size. Comparing the tail behavior in redundancy-$d$ scheduling with the LCFS-PR discipline and with the FCFS discipline we can conclude that, for the c.o.s.\ variant, the LCFS-PR discipline has better tail behavior than (or equally good as) the FCFS discipline. Loosely speaking, the tail behavior of the LCFS-PR discipline is better in scenarios with small load and a small number of replicas $d$ and the tail behavior of the two service disciplines is similar in all other scenarios. 
For the c.o.c.\ variant of the fork-join model the LCFS-PR discipline always has better tail behavior than the FCFS discipline for all dependency structures between the replicas. 

\section{Numerical results}
\label{sec: numerical results}
In the previous sections we determined the tail behavior of the response time for heavy-tailed job sizes. In this section we provide simulation results for redundancy-$d$ scheduling that illustrate this tail behavior in various scenarios. All the simulation experiments are conducted with $10^{9}$ number of jobs. The figures are in log-log scale and we consider Pareto distributed job sizes with shape value $\nu=1.5$, which means that $B \in RV(-1.5)$. Note that in the simulation $\mathbb{P}(R > x) = 0$ for $x$ big enough, which explains the steep drop in all the figures. 

In Figure~\ref{fig: latencytail cos lambda 25} the tail behavior of the response time for the c.o.s.\ variant of redundancy is depicted, see Corollary~\ref{cor: tail waiting response time regularly varying cos fcfs} for the corresponding asymptotic bound. It can be seen that especially the lines for $d=2$ and $d=N=3$ are following the line representing tail index $-0.5$ quite well. For $d=1$ it can be seen that at first it diverges, but after $x>10$ it also runs parallel to the line representing tail index $-0.5$. 

\begin{figure}[H]
  \centering
\resizebox{0.8\textwidth}{!}{\input{TikZFigures/latencyTail_cos_n_10_9_lambda_25_d.tex}}
\caption{Tail behavior for the response time in the c.o.s.\ variant of redundancy-$d$ scheduling with Pareto$(\nu=1.5,x_{m}=1/3)$ job sizes, $\mathbb{E}[B]=1$, $N=3$, $\tilde{\rho}=2.5$ and the FCFS discipline. The dashed line depicts the function $y=x^{-0.5}$.}
\label{fig: latencytail cos lambda 25}
\end{figure}
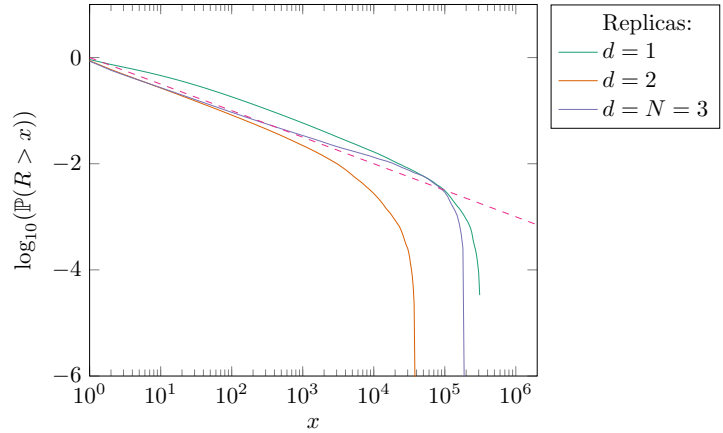

Figure~\ref{fig: latencytail coc identical lambda 05} shows the tail behavior for the response time in the c.o.c.\ variant of redundancy with identical Pareto job sizes, see Corollary~\ref{cor: tail waiting time regularly varying fcfs} for the asymptotic bound. It can be seen that for every number of replicas the tail index is equivalent to the value identified in Corollary~\ref{cor: tail waiting time regularly varying fcfs}. Interestingly, this figure shows that for $d=2$ the asymptotic lower bound represents the exact tail behavior better than the upper bound. 

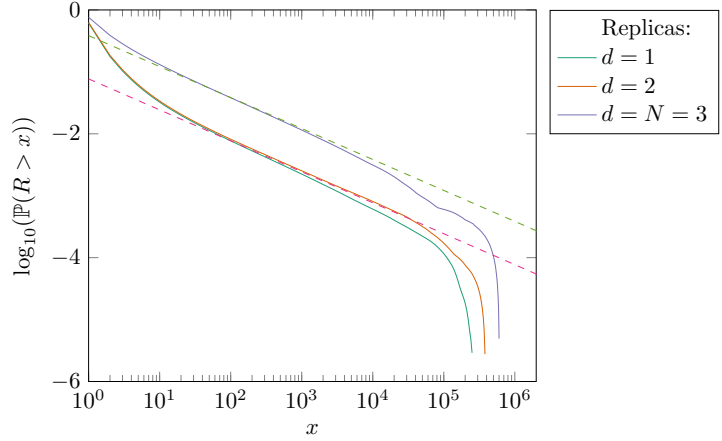
\begin{figure}[H]
  \centering
\resizebox{0.8\textwidth}{!}{\input{TikZFigures/latencyTail_coc_identical_n_10_9_lambda_05_d.tex}}
\caption{Tail behavior for the response time in the c.o.c.\ variant of redundancy-$d$ scheduling with identical Pareto$(\nu=1.5,x_{m}=1/3)$ job sizes, $\mathbb{E}[B]=1$, $N=3$, $\tilde{\rho}=0.5$ and the FCFS discipline. The dashed lines depict the tail behavior for the response time in the lower bound ($\mathbb{P}(R_{\mathrm{L}} > x)$) and in the upper bound ($\mathbb{P}(R_{\mathrm{U}} > x)$) given in Corollary~\ref{cor: tail waiting time regularly varying fcfs}. Note that the system with $d=1$ and $d=3=N$ is equivalent to the lower and upper bound system, respectively.}
\label{fig: latencytail coc identical lambda 05}
\end{figure}

Figure~\ref{fig: latencytail coc iid lambda 05} depicts the tail behavior for the response time in the c.o.c.\ variant of redundancy with i.i.d.\ Pareto job sizes.  Note that according to Corollary~\ref{cor: tail waiting time regularly varying fcfs} the tail index is given by $1-d\nu$. To get the same tail behavior for all the numbers of replicas in Figure~\ref{fig: latencytail coc iid lambda 05} we scaled the job size with $d$. 

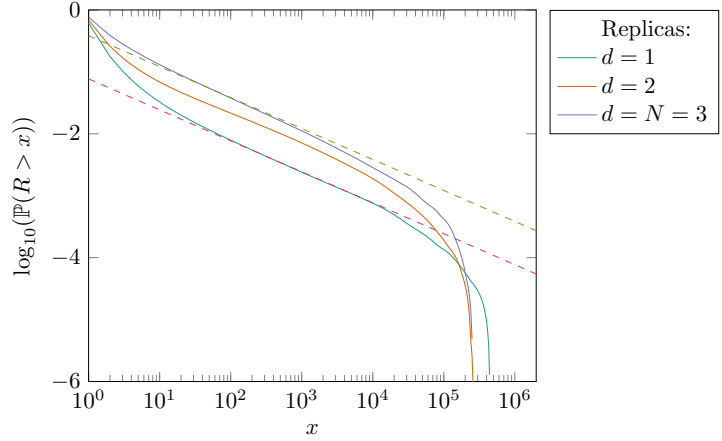
\begin{figure}[H]
  \centering
\resizebox{0.8\textwidth}{!}{\input{TikZFigures/latencyTail_coc_independent_n_10_9_lambda_05_d.tex}}
\caption{Tail behavior for the c.o.c.\ variant of redundancy-$d$ scheduling with i.i.d.\ Pareto$(\nu=1.5/d,x_{m}=1/3)$ job sizes, $\mathbb{E}[B_{\mathrm{min}}]=1$, $N=3$, $\tilde{\rho}=0.5$ and the FCFS discipline. The dashed lines depict the tail behavior for the response time in the lower bound ($\mathbb{P}(R_{\mathrm{L}} > x)$) and in the upper bound ($\mathbb{P}(R_{\mathrm{U}} > x)$) given in Corollary~\ref{cor: tail waiting time regularly varying fcfs}. Note that the system with $d=1$ and $d=3=N$ is equivalent to the lower and upper bound system, respectively.}
\label{fig: latencytail coc iid lambda 05}
\end{figure}

So far we only considered the FCFS discipline. Figure~\ref{fig: latencytail LCFS coc identical lambda 05} shows the tail behavior of the response time for the c.o.c.\ variant of redundancy with the LCFS discipline. 

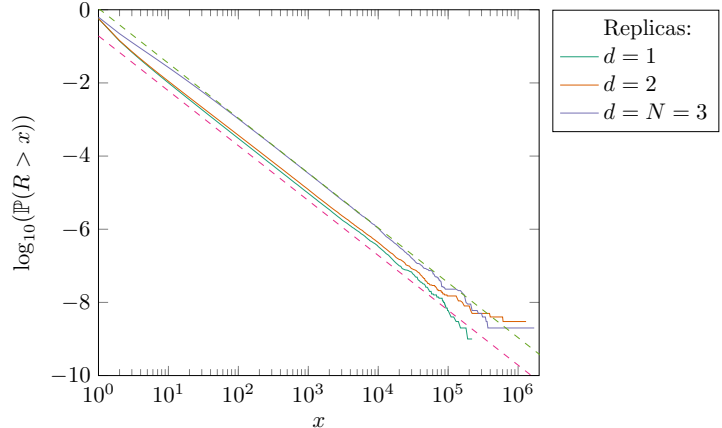
\begin{figure}[H]
  \centering
\resizebox{0.8\textwidth}{!}{\input{TikZFigures/latencyTail_LCFS_coc_identical_n_10_9_lambda_05_d.tex}}
\caption{Tail behavior for the response time in the c.o.c.\ variant of redundancy-$d$ scheduling with identical Pareto$(\nu=1.5,x_{m}=1/3)$ job sizes, $\mathbb{E}[B]=1$, $N=3$, $\tilde{\rho}=0.5$ and the LCFS discipline. The dashed lines depict the tail behavior for the response time in the lower bound ($\mathbb{P}(R_{\mathrm{L}} > x)$) and in the upper bound ($\mathbb{P}(R_{\mathrm{U}} > x)$) given in Theorem~\ref{thm: tail response time regularly varying lcfs}. Note that the system with $d=1$ and $d=3=N$ is equivalent to the lower and upper bound system, respectively.}
\label{fig: latencytail LCFS coc identical lambda 05}
\end{figure}

\section{Conclusion and suggestions for further research}
\label{sec: conclusion}
In this paper we studied the tail behavior of the response time in redundancy-$d$ scheduling and the fork-join model for heavy-tailed job sizes. 
In particular, for the c.o.s.\ variant of redundancy-$d$ with the FCFS discipline and subexponential job sizes we determined the tail behavior of the response time and showed that it depends on the load of the system. 
For the c.o.c.\ variant of the fork-join model we observed that the tail behavior of the response time depends on the dependency structure between the replicas. For job sizes $B \in RV(-\nu)$, our results indicate that for the c.o.s.\ variant of redundancy scheduling in the scenario of sufficiently small load having $d= \lceil \frac{\nu}{\nu-1} \rceil$ replicas already achieves the optimal asymptotic tail behavior of the response time. 
For high loads, the results indicate that creating many replicas yields no benefits for the tail index of the response time.
For the c.o.c.\ variant of the fork-join($n_{\mathrm{F}},n_{\mathrm{J}}$) model with identical and i.i.d.\ replicas the tail index of the response time is $1-\nu$ and $1-(n_{\mathrm{F}}+1-n_{\mathrm{J}})\nu$, respectively. Thus, the tail index is independent of the load of the system and for identical replicas even independent of the number of replicas.

Observe that all the results on the tail index for the c.o.c.\ variant of the fork-join model rely on the fact that the upper bound system is stable. The stability condition of this system does not necessarily coincide with the stability condition of the original fork-join model. For further research one could study the tail index for these values for the load, i.e., whenever the original fork-join model is stable but the upper bound system is unstable. 

A natural topic for further research would be to extend our analysis to heterogeneous servers or even more generally to job types that can have different speeds at the various servers, see for example the model in~\cite{RBB-DPP}.
 
Another extension would be to analyze the tail behavior of the response time for the ROS service discipline. As mentioned in the introduction, for the single-server queue this discipline has the same tail index as the FCFS discipline. Simulation experiments (not included in this paper) suggest that this statement extends to redundancy-$d$ scheduling. 
 
\section*{Acknowledgments}
The work in this paper is supported by the Netherlands Organisation for Scientific Research (NWO) through Gravitation grant NETWORKS 024.002.003.
The authors would like to thank Sergey Foss for his helpful suggestions and careful reading. 

\bibliography{references}

\appendix

\section{Preliminary results}
\label{app sec: preliminary results}

In this appendix we introduce several classes of heavy-tailed distributions that are considered for the job size in this paper, see also~\cite{BGT-RV,FKZ-IHTSED}. 
Let the complementary cumulative distribution function be defined as $\bar{F}_{B}(x):= 1-F_{B}(x) = \mathbb{P}(B > x)$.
\begin{definition} 
$B$ is heavy-tailed if, for all $\epsilon>0$,
\begin{align*}
\mathbb{E}[e^{\epsilon B}] = \infty,
\end{align*}
or equivalently (see for example~\cite[Theorem 2.6]{FKZ-IHTSED}), if for all $\epsilon>0$,
\begin{align*}
\mathbb{P}(B > x) e^{\epsilon x} \rightarrow \infty \text{ as } x \rightarrow \infty
\end{align*}
\end{definition}

Let $F_{B}^{n*}(x)$ be the $n$-fold convolution of $F_{B}(x)$ for $n=2,3,\dots$, with $F_B^{1*}(x)~\equiv F_B(x)$.

\begin{definition}
\label{def: subexponential}
$B$ is subexponential, denoted by $B \in \mathcal{S}$, if
\begin{align*}
\frac{\bar{F}_{B}^{2*}(x)}{\bar{F}_{B}(x)} = \frac{\mathbb{P}(B_{1}+B_{2}>x)}{\mathbb{P}(B>x)} \rightarrow 2 ~~~ \text{as }x \rightarrow \infty.
\end{align*}
\end{definition}

Examples of well-known subexponential distributions are Pareto, Lognormal and Weibull with a shape parameter between $0$ and $1$.

\begin{definition}
\label{def: long tailed}
$B$ is long-tailed, denoted by $B \in \mathcal{L}$, if $\bar{F}_{B}(x+1) \sim \bar{F}_{B}(x)$ as $x \rightarrow \infty$.
\end{definition}

Here $f(x) \sim g(x)$ means $\frac{f(x)}{g(x)} \rightarrow 1$ as $x \rightarrow \infty$.

\begin{definition}
\label{def: dominated varying}
$B$ is dominated varying, denoted by $B \in \mathcal{D}$, if $\bar{F}_{B}(2x) \geq c \bar{F}_{B}(x)$ for some $c > 0$ and for all $x$.
\end{definition}

\begin{definition}
$B$ is $\mathcal{O}$-regularly varying, denoted by $B \in ORV$, if 
\begin{align*}
0 < \liminf_{x \rightarrow \infty} \frac{\bar{F}_{B}(\alpha x )}{\bar{F}_{B}(x)} \leq \limsup_{x \rightarrow \infty} \frac{\bar{F}_{B}(\alpha x)}{\bar{F}_{B}(x)} < \infty, ~~~ \forall \alpha \geq 1.
\end{align*}
Furthermore, $B \in ORV(-\nu)$ if
\begin{align}
\label{eq: definition o regularly varying}
c_{1} \alpha^{-\nu} < \liminf_{x \rightarrow \infty} \frac{\bar{F}_{B}(\alpha x )}{\bar{F}_{B}(x)} \leq \limsup_{x \rightarrow \infty} \frac{\bar{F}_{B}(\alpha x)}{\bar{F}_{B}(x)} < c_{2} \alpha^{-\nu}, ~~~ \forall \alpha \geq 1,
\end{align}
with positive constants $c_{1}$ and $c_{2}$.
\end{definition}

\begin{definition}
\label{def: regularly varying}
$B$ is regularly varying of index $-\nu$, denoted by $B \in RV(-\nu)$, if
\begin{equation}
\bar{F}_{B}(x) = L(x) x^{-\nu}, ~~~x>0,
\end{equation}
with $L(x)$ a slowly varying function, i.e., $L(\alpha x)/L(x) \rightarrow 1$ for any $\alpha >0$.
\end{definition}
Observe that we have the following relations, see for example~\cite[Theorem 2.1.8]{BGT-RV} or~\cite[Chapter 2]{Z-QSHT},
\begin{align}
\label{eq: relation tail classes}
RV \subset \mathcal{D} \subset ORV  ~~~ \text{and} ~~~ RV \subset \mathcal{L} \cap \mathcal{D} \subset \mathcal{S}.
\end{align}

We will analyze the tail asymptotics of the response time of an arbitrary job in steady state.
For this we need some preliminary results that are stated in the lemmas below. 
The next lemma states that the minimum and the sum of two independent regularly varying random variables is again regularly varying.  

\begin{lemma}
\label{lem: sum minimum regularly varying}
Let $X$ and $Y$ be two independent regularly varying random variables with $\mathbb{P}(X>x)=L_{1}(x)x^{-\nu_{1}}$ and $\mathbb{P}(Y>x)=L_{2}(x)x^{-\nu_{2}}$. Then,
\begin{enumerate}[i)]
\item  $\min \left( X,Y \right) \in RV(-(\nu_{1}+\nu_{2}))$,
\item  $X+Y \in RV(-\min\{\nu_{1},\nu_{2}\})$.
\end{enumerate}
\end{lemma}
\begin{proof}
See~\cite[Proposition~1.5.7]{BGT-RV}.
\end{proof}

A similar lemma as Lemma~\ref{lem: sum minimum regularly varying} can be proved for $\mathcal{O}$-regularly varying random variables. 

\begin{lemma}
\label{lem: sum o regularly varying}
Let $X$ and $Y$ be two independent $\mathcal{O}$-regularly varying random variables with index $-\nu_{1}$ and $-\nu_{2}$, respectively, then
\begin{enumerate}[i)]
\item $\min \left(X,Y \right) \in ORV(-(\nu_{1}+\nu_{2}))$,
\item $X+Y \in ORV(-\min\{\nu_{1},\nu_{2}\})$.
\end{enumerate}
\end{lemma}
\begin{proof}
By definition of $ORV(-\nu)$ there exist $c_{i}$, $i=1,\dots,4$, such that
\begin{align*}
c_{1} \alpha^{-\nu_{1}} < \liminf_{x \rightarrow \infty} \frac{\bar{F}_{X}(\alpha x )}{\bar{F}_{X}(x)} \leq \limsup_{x \rightarrow \infty} \frac{\bar{F}_{X}(\alpha x)}{\bar{F}_{X}(x)} < c_{2} \alpha^{-\nu_{1}}, ~~~ \forall \alpha \geq 1,
\end{align*}
and
\begin{align*}
c_{3} \alpha^{-\nu_{2}} < \liminf_{x \rightarrow \infty} \frac{\bar{F}_{Y}(\alpha x )}{\bar{F}_{Y}(x)} \leq \limsup_{x \rightarrow \infty} \frac{\bar{F}_{Y}(\alpha x)}{\bar{F}_{Y}(x)} < c_{4} \alpha^{-\nu_{2}}, ~~~ \forall \alpha \geq 1.
\end{align*}
Observe that by independence we have
\begin{align*}
\mathbb{P}(\min \left(X,Y \right) > x) = \mathbb{P}(X > x)\mathbb{P}(Y > x),
\end{align*}
and therefore
\begin{align*}
c_{1}c_{3} \alpha^{-(\nu_{1}+\nu_{2})} < \liminf_{x \rightarrow \infty} \frac{\bar{F}_{\min \left(X,Y \right)}(\alpha x )}{\bar{F}_{\min \left(X,Y \right)}(x)} \leq \limsup_{x \rightarrow \infty} \frac{\bar{F}_{\min \left(X,Y \right)}(\alpha x)}{\bar{F}_{\min \left(X,Y \right)}(x)} < c_{2}c_{4} \alpha^{-(\nu_{1}+\nu_{2})}, ~~~ \forall \alpha \geq 1.
\end{align*}
The proof of i) follows by the definition of $ORV(-\nu)$. 

The proof of ii) is similar to the proof of the convolution closure for regularly varying distributions, see for example~\cite{M-RVSAP}.
Since $\{X+Y > x\} \supset \{X>x\} \cup \{Y>x\}$ it follows that
\begin{align*}
\mathbb{P}(X+Y > x) \geq  \mathbb{P}(X>x)+\mathbb{P}(Y>x) - \mathbb{P}(X > x)\mathbb{P}(Y > x).
\end{align*}
For $0 < \delta < \frac{1}{2}$, we have that
\begin{align*}
\{X+Y > x\} \subset \{X > (1-\delta)x \} \cup \{Y > (1-\delta)x \} \cup \{X> \delta x, Y > \delta x\},
\end{align*}
and therefore
\begin{align*}
\mathbb{P}(X+Y > x) \leq \mathbb{P}(X > (1-\delta)x) + \mathbb{P}(Y > (1-\delta)x) + \mathbb{P}(X > \delta x)\mathbb{P}(Y > \delta x).
\end{align*}
Now if $\nu_{1} = \min \{v_{1}, v_{2}\}$,
\begin{align*}
\frac{\mathbb{P}(X+Y > \alpha x)}{\mathbb{P}(X+Y > x)} \geq \frac{c_{1}}{c_{2}} (1-\delta)^{\nu_{1}} \alpha^{-\nu_{1}} (1+o(1)),
\end{align*}
and
\begin{align*}
\frac{\mathbb{P}(X+Y > \alpha x)}{\mathbb{P}(X+Y > x)} \leq \frac{c_{2}}{c_{1}} (1-\delta)^{-\nu_{1}} \alpha^{-\nu_{1}} (1+o(1)).
\end{align*}
The case for $\nu_{2} = \min \{v_{1}, v_{2}\}$ follows by an analogous argument. 
By definition, see~\eqref{eq: definition o regularly varying}, we get that $X+Y \in ORV(-\min\{\nu_{1},\nu_{2}\})$.
\end{proof}

Observe that these results could also be obtained by applying the principle of a single big jump, see for example~\cite{FR-SSRV}.

Next we give an auxiliary lemma which states that a random variable is $\mathcal{O}$-regularly varying with index $-\nu$ under the condition $c_{1} L(x) x^{-\nu} \leq \mathbb{P}(X > x) \leq c_{2} L(x) x^{-\nu}$. 

\begin{lemma}
\label{lem: bounds in o regularly varying}
If $c_{1} L(x) x^{-\nu} \leq \mathbb{P}(X > x) \leq c_{2} L(x) x^{-\nu}$, then $X \in ORV(-\nu)$.
\end{lemma}
\begin{proof}
From $c_{1} L(x) x^{-\nu} \leq \mathbb{P}(X > x) \leq c_{2} L(x) x^{-\nu}$ we get
\begin{align*}
\frac{c_{1}}{c_{2}} \alpha^{-\nu} < \liminf_{x \rightarrow \infty} \frac{\bar{F}_{X}(\alpha x )}{\bar{F}_{X}(x)} \leq \limsup_{x \rightarrow \infty} \frac{\bar{F}_{X}(\alpha x)}{\bar{F}_{X}(x)} < \frac{c_{2}}{c_{1}} \alpha^{-\nu}, ~~~ \forall \alpha \geq 1.
\end{align*}
The proof follows by definition of $ORV(-\nu)$.
\end{proof}

\section{Proof of the upper bounds in Theorem~\ref{thm: tail waiting time cos fcfs}}
\label{app sec: proof of theorem cos fcfs}
In this appendix we will prove the upper bounds~\eqref{eq: thm upperbound waiting time cos fcfs low load} and~\eqref{eq: thm upperbound waiting time cos fcfs high load} on the tail of the waiting time for the c.o.s.\ variant of redundancy-$d$ with the FCFS discipline.
Our proof is based on the proof in~\cite{FK-LDMSHV} for the $GI/G/N$ queue, which corresponds to a system of $N$ queues with the JSW-$N$ policy. 
While the JSW-$d$ policy with $1 \leq d \leq N$ requires essential and sometimes subtle adaptations, overall we follow the main line of reasoning of~\cite{FK-LDMSHV} and indicate for each lemma and theorem its counterpart in~\cite{FK-LDMSHV}.

In most heavy-tail results in queueing theory, the interarrival time distribution does not have an effect on the waiting time tail behavior. With this in mind, similar to~\cite{FK-HTMS,FK-LDMSHV}, we first consider deterministic interarrival times, making the derivations more tractable, and thereafter prove Lemma~\ref{app lem: lemma 1} which allows us to extend the proof for deterministic interarrival times to generally distributed interarrival times.
The idea behind the proof of the upper bounds is that we compare the system with the JSW-$d$ policy to $N$ auxiliary single-server queueing systems which work in parallel. 

Let $\boldsymbol{s}^{n}$ denote the vector of $d$ servers which are sampled for the $n$th job. For $n=1,2,\dots$, let $\boldsymbol{V}_{n}=(V_{n1},\dots,V_{nN})$ be the vector of residual workloads at the arrival epoch of the $n$th job. The waiting time that the $n$th job experiences is $W_{\mathrm{min},n} := \min \{V_{nj},j \in \boldsymbol{s}^{n}\}$ and it joins server $i_{n}$, where $i_{n} = \min\{ i \in \boldsymbol{s}^{n} : V_{ni} = W_{\mathrm{min},n} \}$.
Also,
\begin{align*}
V_{n+1,i} =
\begin{cases}
(V_{ni}+b_{n}-a_{n+1})^{+} & \text{if } i = i_{n}, \\
(V_{ni}-a_{n+1})^{+} & \text{if } i \neq i_{n},
\end{cases}
\end{align*}
with job sizes $b_{n}$ and interarrival times $a_{n}$
Let $R(\boldsymbol{x}) = (R_{1}(\boldsymbol{x}),\dots,R_{N}(\boldsymbol{x}))$ be the operator on $\mathbb{R}^{N}$ that orders the coordinates of $\boldsymbol{x} \in \mathbb{R}^{N}$ in nondecreasing order, i.e., $R_{1}(\boldsymbol{x})\leq \dots \leq R_{N}(\boldsymbol{x})$. Moreover, let $f_{R}: \mathbb{R} \rightarrow \mathbb{R}$ be the function that maps the server number to the number ordered by workload as in the operator $R(\cdot)$. For $n=1,2,\dots$, put $\boldsymbol{D}_{n} = R(\boldsymbol{V}_{n})$. Then $W_{\mathrm{min},n} = D_{ni}$, where $i = f_{R}(i_{n})$ and similar to the Kiefer-Wolfowitz recursion for the JSW policy, we get
\begin{align}
\label{eq: kiefer wolfowitz jsw-d}
\boldsymbol{D}_{n+1} = R((D_{n1}-a_{n+1})^{+},\dots,(D_{ni}+b_{n}-a_{n+1})^{+},\dots,(D_{nN}-a_{n+1})^{+}).
\end{align}
Observe that the operator $R(\cdot)$ is monotone, thus the sequence $\boldsymbol{D}_{n}$ satisfying Equation~\eqref{eq: kiefer wolfowitz jsw-d} satisfies the two monotonicity properties of Lemma 4.1 in~\cite{FK-LDMSHV} as well.

Hereafter we continue to assume deterministic interarrival times $a' \equiv \mathbb{E}[A]$ for both the original system and the auxiliary systems introduced below.
Consider $N$ auxiliary $D/G/1$ queueing systems which work in parallel. At every time instant $T_{n}$, $n=1,2,\dots$, a batch of $N$ jobs arrives, one job per server. 
Denote by $U_{ni}$, $i=1,\dots,N$, the waiting times in the $i$th $D/G/1$ queue and let $b_{ni}$, $n \geq 1$ and $i=1,\dots,N$, be independent random variables with common distribution $B$.
We couple the job sizes of the $D/G/N$ redundancy-$d$ system with job sizes at the $N$ auxiliary $D/G/1$ queues: we let $b_{n} = b_{n,i_{n}}$, where $i_{n} = \min \{ i \in \boldsymbol{s}^{n}: V_{ni} = W_{\mathrm{min},n} \}$ as defined earlier. 
The deterministic interarrival times are $T_{n} = n (k+1) (a' - h)$, with
\begin{align}
\label{app eq: equation h upperbound}
\frac{k}{k+1} \left( a' - \frac{\mathbb{E}[B]}{k+1} \right) < h < a' - \frac{\mathbb{E}[B]}{k+1},
\end{align} 
so that the auxiliary queueing systems are stable.   

In Lemma~\ref{app lem: lemma 6.2} we upper bound the sum of waiting times by the sum of waiting times in the auxiliary $D/G/1$ queues and a light-tailed random variable. The proof of Lemma~\ref{app lem: lemma 6.2} uses the auxiliary Lemma~\ref{app lem: lemma 4.3} that provides an upper bound on the expected difference of the total workload at all the servers seen by the first and $(s+1)$th job when the workload at one of the servers is large. Note that the choice of this large workload is different from that used in~\cite{FK-LDMSHV}.

\begin{lemma}[Counterpart of Lemma~4.3 in~\cite{FK-LDMSHV}]
\label{app lem: lemma 4.3}
Consider a system with $k+1$ servers and assume $\mathbb{E}[B] > k a'$. For any $\epsilon > 0$, there exist $V_{\mathrm{large}} < \infty$ and an integer $s \geq 1$ such that, for any initial value $\boldsymbol{D}_{1}$ with $D_{1,k+1} \geq V_{\mathrm{large}}$,
\begin{align*}
\mathbb{E}\left[ \sum_{j=1}^{k+1} D_{1+s,j} - \sum_{j=1}^{k+1} D_{1j} \right] \leq s(\mathbb{E}[B] - (k+1)a' + \epsilon).
\end{align*}
\end{lemma}
\begin{proof} 
By property (2) of Lemma~4.1 in~\cite{FK-LDMSHV}, it is enough to prove the result for initial values $D_{11} = \dots = D_{1k} = 0$, $D_{1,k+1} = V_{\mathrm{large}}$ only.
Choose $C$ such that $\mathbb{E}[\min \{a', C\}] \geq a' - \epsilon/2$. By property (1) of Lemma~4.1 in~\cite{FK-LDMSHV}, we may prove the lemma with interarrival times $\min\{a', C\}$ instead of $a'$.

For $d \geq k+1$ the proof follows from Lemma~4.3 in~\cite{FK-LDMSHV}, since the JSW-$d$ and JSW policy are equivalent in the system with $k+1$ servers. 
For $d < k+1$, consider an auxiliary unstable $GI/G/k$ system with initial value $\boldsymbol{\hat{D}}_{1}=0$ and find $s$ such that $\mathbb{E}\left[\sum_{i=1}^{k} \hat{D}_{1+s,i}\right] \leq s(\mathbb{E}[B]-k a' + \epsilon/2)$. Note that this system samples $d-1$ servers with probability $\frac{d}{k+1}$, i.e., the probability that server $k+1$ is sampled in the original system, and $d$ servers with probability $1-\frac{d}{k+1}$, i.e., the probability that server $k+1$ is not sampled in the original system. For an unstable system with workload vector $\hat{\boldsymbol{D}}_{n}$ we have that $\hat{D}_{ni} \rightarrow \infty$ as $n \rightarrow \infty$ for $i=1,\dots,k$. 

Take $V_{\mathrm{large}}=\max\{(s+1)C,V_{\mathrm{large}}^{*}\}$, where $V_{\mathrm{large}}^{*}$ is defined as follows. Consider the system with initial values $D_{11} = \dots = D_{1k} = 0$, $D_{1,k+1} = V_{\mathrm{large}}^{*}$ and let the $n^{*}$th job be the first job that is assigned to the $(k+1)$th server, i.e., $n^{*} :=\min\{n \geq 1 : i_{n}=k+1\}$ which clearly depends on the initial workload $V_{\mathrm{large}}^{*}$. Then take $V_{\mathrm{large}}^{*}$ such that
\begin{align*}
\min_{i=1,\dots,k+1} D_{n^{*}i} \geq \left(s+1-n^{*}\right)C.
\end{align*}
Note that such $V_{\mathrm{large}}^{*}$ exists, since increasing $V_{\mathrm{large}}^{*}$ leads, loosely speaking, to increasing workloads at the other $k$ servers as well (because they are unstable). 
This definition of $V_{\mathrm{large}}^{*}$ ensures that the first time a job is allocated to server $k+1$ the workload at the other servers is large enough so that, without any additional work, these servers are not empty before the $(s+1)$th job.
We cannot simply take $V_{\mathrm{large}}=(s+1)C$ as in~\cite{FK-LDMSHV}, because this does not guarantee that $D_{ni} > 0$ and $\hat{D}_{ni} > 0$, for all $n \in [\min\{s+1, n^{*}\},s+1]$ and $i=1,\dots,k$, which is needed in the proof. Indeed, without additional constraints on $V_{\mathrm{large}}$ it may be that the job is allocated to the $(k+1)$th server, which has the smallest workload out of the $d$ sampled servers, while at least one of the other $k+1-d$ servers is empty.

By the exact same steps as in Lemma 4.3 in~\cite{FK-LDMSHV} we can prove that
\begin{align}
\label{eq: equality unstable system}
\sum_{j=1}^{k+1} D_{1+s,j} - \sum_{j=1}^{k+1} D_{1,j} &= \sum_{j=1}^{k} \hat{D}_{1+s,j} - \sum_{j=1}^{s} \min\{a',C\} \nonumber \\
& \leq s(\mathbb{E}[B] - k a' + \frac{\epsilon}{2}) - s(a' - \frac{\epsilon}{2})~~~ \text{a.s.},
\end{align}
and the result follows.
\end{proof}

\begin{lemma}[Counterpart of Lemma 6.2 in~\cite{FK-LDMSHV}]
\label{app lem: lemma 6.2}
There exists $\beta > 0$ such that, for any set of $k+1$ indices $I=\{i(1),\dots,i(k+1)\}$, there is a random variable $\eta_{I}$ such that $\mathbb{E}[e^{\beta \eta_{I}}] < \infty$ and, for any $n$, with probability~$1$,
\begin{align}
\label{eq: bound workload auxiliary system light tailed}
\sum_{i \in I} V_{ni} \leq \sum_{i \in I} U_{ni} + \eta_{I}.
\end{align}
\end{lemma}
\begin{proof}
Fix some $i^{*} \in I$. Observe that for $d \geq k+1$ the proof directly follows from Lemma 6.2 in~\cite{FK-LDMSHV}, since the JSW-$d$ and JSW policy are equivalent in the system with $k+1$ servers. For $d < k+1$, consider an auxiliary $GI/G/(k+1)$ redundancy-$d$ system as in Lemma~\ref{app lem: lemma 4.3} with workloads $V^{*}_{n}=(V^{*}_{ni}, i \in I)$ with the same interarrival times equal to $a'$, but whose service times $b^{*}_{n}$ are chosen in a special manner. At any time $n$, if $i_{n} \in I$, then put $b^{*}_{n} = b_{n,i_{n}}$ and $i^{*}_{n} = i_{n}$. If $i^{*}_{n} \notin I$, then put $b^{*}_{n} = b_{n,i^{*}}$ and $i^{*}_{n} = i^{*}$. Applying property (1) of Lemma~4.1 in~\cite{FK-LDMSHV}, we get that $R(V_{ni}, i \in I) \leq R(\boldsymbol{V}^{*}_{n})$ coordinate-wise, for any $n$. 
Therefore,
\begin{align*}
\sum_{i \in I} V_{ni} \leq \sum_{i \in I} V^{*}_{ni}.
\end{align*}
By the exact same steps as in Lemma 6.2 in~\cite{FK-LDMSHV} using Lemma~\ref{app lem: lemma 4.3} (the counterpart of Lemma~4.3 in~\cite{FK-LDMSHV}) we can prove Equation~\eqref{eq: bound workload auxiliary system light tailed}.
\end{proof}

Just like a crucial step in~\cite{FK-LDMSHV}, Lemma~\ref{app lem: lemma 6.2} can be used to upper bound the waiting time in the $N$-server system by the waiting time in the corresponding system with \textit{deterministic} interarrival times minus a negligible term. Note that the upper bounds are not as sharp as in~\cite{FK-LDMSHV} since, unlike~\cite{FK-LDMSHV},
$W_{\mathrm{min},n} \nleq \frac{1}{k+1} \sum_{i \in I} V_{ni}$ for every collection $I$. 

\begin{lemma}[Counterpart of Lemma 6.1 in~\cite{FK-LDMSHV}]
\label{app lem: lemma 6.1}
There exists a number $\beta > 0$ and a random variable $\eta$ such that $\mathbb{E}[e^{\beta \eta}] < \infty$ and, for all $n$, with probability~$1$\\
i) if $k \geq N-d$,
\begin{align*}
W_{\mathrm{min},n} \leq \frac{k+1}{k+1-N+d}U_{n,(k+1)} + \eta,
\end{align*}
where $U_{n,(k+1)}$ is the $(k+1)$th order statistic of vector $(U_{n1}, \dots, U_{nN})$,\\
ii) if $k \leq N-d$,
\begin{align*}
W_{\mathrm{min},n} \leq (k+1)U_{n,(N-d+1)} + \eta.
\end{align*}
\end{lemma}
\begin{proof}
i) For $k \geq N-d$ we have for every collection $I$ of $k+1$ coordinates,
\begin{align}
\label{eq: upperbound waiting time cos}
W_{\mathrm{min},n} \leq \frac{1}{k+1-N+d} \sum_{i \in I} V_{ni},
\end{align}
since $W_{\mathrm{min},n}$ is no larger than the $(N-d+1)$th smallest value of $V_{ni}$, $i \in I$. 
Then it follows from Lemma~\ref{app lem: lemma 6.2} that
\begin{align}
\label{eq: upperbound waiting time auxiliary cos}
W_{\mathrm{min},n} \leq \frac{1}{k+1-N+d} \sum_{i \in I} U_{ni} + \eta,
\end{align}
where $\eta := \max_{I : |I|=k+1} \eta_{I}$.
Take $I$ such that $\{U_{ni}, i \in I\}$ are the $k+1$ smallest coordinates of the vector $(U_{n1},\dots,U_{nN})$. Then $U_{ni} \leq U_{n,(k+1)}$ for every $i \in I$.

ii) For $k \leq N-d$ we take the collection $I$ of $k+1$ coordinates such that
$\tilde{i}_{n} = \text{arg}\min_{i}\{U_{ni} : i \in \boldsymbol{s}^{n}\} \in I$. Hence, $I \cap \boldsymbol{s}^{n} \neq \emptyset$ and again Equations~\eqref{eq: upperbound waiting time cos} and~\eqref{eq: upperbound waiting time auxiliary cos} hold. 
Take the remaining coordinates of $I$ such that $\{U_{ni}, i \in I \setminus \tilde{i}_{n}\}$ are the $k$ smallest coordinates of the vector $(U_{n1},\dots,U_{n,\tilde{i}_{n}-1},U_{n,\tilde{i}_{n}+1},\dots,U_{nN})$. Then $U_{ni} \leq U_{n,(N-d+1)}$ for every $i \in I$. Indeed, in the worst case $\{U_{ni} : i \in \boldsymbol{s}^{n}\}$ are the $d$ largest coordinates of the vector $(U_{n1},\dots,U_{nN})$, but $\tilde{i}_{n}$ is defined as the argument that achieves the minimum of the set $\{U_{ni} : i \in \boldsymbol{s}^{n}\}$ from which it follows that $U_{n\tilde{i}_{n}} \leq U_{n,(N-d+1)}$.
\end{proof}

\begin{theorem}[Analogous to Theorem 7.1 in~\cite{FK-LDMSHV}]
\label{app thm: upperbound tail waiting time}
Let $\tilde{\rho}<k+1$ for some $k \in \{0,\dots,N-1\}$. Then for any fixed $h$ satisfying Equation~\eqref{app eq: equation h upperbound} there exists $\beta > 0$ such that\\
i) if $k \geq N-d$,
\begin{align*}
\mathbb{P}(W_{\mathrm{min}} > x+y) \leq {N \choose k} \left(\bar{F}_{M_{\mathrm{rw}}}\left(\frac{(k+1-N+d)x}{k+1}\right)\right)^{N-k} + \text{const} \cdot e^{-\beta y},
\end{align*} 
for all $x,y > 0$,\\
ii) if $k \leq N-d$,
\begin{align*}
\mathbb{P}(W_{\mathrm{min}} > x+y) \leq {N \choose d} \left(\bar{F}_{M_{\mathrm{rw}}}\left(\frac{x}{k+1} \right) \right)^{d} + \text{const} \cdot e^{-\beta y}, 
\end{align*}
for all $x,y > 0$, where $F_{M_{\mathrm{rw}}}$ is the cumulative distribution function of the random variable
\begin{align*}
M_{\mathrm{rw}} := \sup_{n \geq 1} \left\{ 0, \sum_{j=1}^{n}(b_{j}-(k+1)(a'-h))\right\}.
\end{align*}
\end{theorem}
\begin{proof}
Similarly to~\cite[Theorem~7.1]{FK-LDMSHV}, this proof relies on Lemma~\ref{app lem: lemma 6.1} which upper bounds the waiting time of the $n$th job in the two cases $k \geq N-d$ and $k \leq N-d$. 

Consider deterministic arrival times $T_{n} = n (k+1)(a'-h)$ where $h$ satisfies~\eqref{app eq: equation h upperbound}.
For $k \geq N-d$, by Lemma~\ref{app lem: lemma 6.1} i),
\begin{align*}
\mathbb{P}(W_{\mathrm{min},n} > x + y) &\leq \mathbb{P}\left(U_{n,(k+1)} > \frac{(k+1-N+d)x}{k+1}\right) + \mathbb{P}(\eta > y),
\end{align*}
and taking into account the independence of the auxiliary queueing systems, we obtain
\begin{align*}
\mathbb{P}(W_{\mathrm{min},n} > x + y) &\leq {N \choose k} \left( \mathbb{P}\left(U_{n1} > \frac{(k+1-N+d)x}{k+1}\right) \right)^{N-k} + \mathbb{P}(\eta > y).
\end{align*}
Similarly, for $k \leq N-d$, by Lemma~\ref{app lem: lemma 6.1} ii),
\begin{align*}
\mathbb{P}(W_{\mathrm{min},n} > x + y) &\leq \mathbb{P}\left(U_{n,(N-d+1)} > \frac{x}{k+1}\right) + \mathbb{P}(\eta > y),
\end{align*}
and taking into account the independence of the auxiliary queueing systems, we obtain
\begin{align*}
\mathbb{P}(W_{\mathrm{min},n} > x + y) &\leq {N \choose d} \left( \mathbb{P}\left(U_{n1} > \frac{x}{k+1}\right) \right)^{d} + \mathbb{P}(\eta > y).
\end{align*}
The proof is completed by observing (1) that $M_{\mathrm{rw}}$ has the distribution of the maximum of a random walk, which is also the distribution of the steady-state waiting time $U$ in any of the auxiliary $D/G/1$ queues, and (2) according to Lemma~\ref{app lem: lemma 6.1}, $\mathbb{P}(\eta > y)$ is exponentially bounded.  
\end{proof}

It is well known (cf.~\cite[Theorem 5.2]{FKZ-IHTSED}) that if the residual job size $B^{\mathrm{res}}$ is subexponential, then
\begin{align*}
\bar{F}_{M_{\mathrm{rw}}}(x) \sim \frac{\mathbb{E}[B]}{(k+1)(a'-h)-\mathbb{E}[B]} \bar{B}^{\mathrm{res}}(x) ~~~ \text{as } x \rightarrow \infty.
\end{align*}

Taking $h$ close to its minimal value, we arrive at the following estimate
\begin{align}
\label{app eq: behavior maximum random walk}
\bar{F}_{M_{\mathrm{rw}}}(x) \sim \left(\frac{(k+1) \tilde{\rho}}{(k+1)-\tilde{\rho} - \frac{(k+1)^{2} \epsilon}{a'}} \right) \bar{B}^{\mathrm{res}}(x) ~~~ \text{as } x \rightarrow \infty, 
\end{align}
where $\epsilon >0$.
We now have the ingredients to prove the upper bounds~\eqref{eq: thm upperbound waiting time cos fcfs low load} and~\eqref{eq: thm upperbound waiting time cos fcfs high load} in Theorem~\ref{thm: tail waiting time cos fcfs}. 
Replacing $x$ by $(1-\delta)x$ and $y$ by $\delta x$ in Theorem~\ref{app thm: upperbound tail waiting time} results in 
\begin{align*}
\mathbb{P}(W_{\mathrm{min}} > x) \leq {N \choose k} \left( \bar{F}_{M_{\mathrm{rw}}}\left(\frac{(k+1-N+d)x(1-\delta)}{k+1}\right) \right)^{N-k} + const \cdot e ^{-\beta \delta x},
\end{align*}
and
\begin{align*}
\mathbb{P}(W_{\mathrm{min}} > x) \leq {N \choose d} \left( \bar{F}_{M_{\mathrm{rw}}}\left(\frac{x(1-\delta)}{k+1} \right) \right)^{d} + const \cdot e ^{-\beta \delta x},
\end{align*}
for $k \geq N-d$ and $k \leq N-d$, respectively. Combined with Equation~\eqref{app eq: behavior maximum random walk} this yields the upper bound. 

So far we assumed deterministic interarrival times; the following lemma allows us to extend the proof to the case of generally distributed interarrival times. For clarity we highlight the metrics that correspond to the system with deterministic interarrival times by an apostrophe.

\begin{lemma}[Counterpart of Lemma~1 in~\cite{FK-HTMS}]
\label{app lem: lemma 1}
If $\mathbb{P}(W_{\mathrm{min}}' > x) \leq \bar{G}(x)$ for some long-tailed distribution $G$, where $W_{\mathrm{min}}'$ denotes the waiting time in the system with deterministic interarrival times $a' \equiv \mathbb{E}[A]$, then 
\begin{align*}
\limsup_{x \rightarrow \infty} \frac{\mathbb{P}(W_{\mathrm{min}} > x)}{\bar{G}(x)} \leq 1.
\end{align*}
\end{lemma}
\begin{proof} 
Denote $\xi_{n} = a' - a_{n}$. Put $M_{0}=0$ and, for $n \geq 1$,
\begin{align*}
M_{n} &= \max\{0,\xi_{n},\xi_{n}+\xi_{n-1},\dots,\xi_{n}+\dots+\xi_{1}\} = (\xi_{n} + M_{n-1})^{+}.
\end{align*}
Similar to Lemma~1 in~\cite{FK-HTMS} we use induction to prove the inequality
\begin{align}
\label{eq: induction waiting time}
W_{\mathrm{min},n} \leq W_{\mathrm{min},n}' + i M_{n} ~~~\text{a.s.}
\end{align}
For $n=1$ we have $0 \leq 0 + i M_{1}$. Assume the inequality is proved for some $n$, then
\begin{align*}
W_{\mathrm{min},n+1} &= R(W_{\mathrm{min},n} + e_{i_{n}}b_{n} - i a_{n+1})^{+} \\
& \leq R(W_{\mathrm{min},n}' + i M_{n} + e_{i_{n}}b_{n} - i a_{n+1})^{+} \\
& = R(W_{\mathrm{min},n}' + e_{i_{n}}b_{n} -ia' + i(M_{n} + \xi_{n+1}))^{+}.
\end{align*}
Since $(u+v)^{+} \leq u^{+} + v^{+}$,
\begin{align*}
W_{\mathrm{min},n+1} \leq  R(W_{\mathrm{min},n}' + e_{i_{n}}b_{n} -ia')^{+} + i(M_{n} + \xi_{n+1})^{+} \equiv W_{\mathrm{min},n+1}' + iM_{n+1},
\end{align*}
and the proof of~\eqref{eq: induction waiting time} is complete. The remainder of the proof follows by the exact same steps as in Lemma~1 in~\cite{FK-HTMS}. 
\end{proof}
\end{document}

%% file: TikZFigures/latencyTail_cos_n_10_9_lambda_25_d.tex
\newcommand{\dataFigure}{TikZFigures/latencyTail_cos_n_10_9_lambda_25_d.csv}

\begin{tikzpicture}
		\begin{axis}[
			xlabel=$x$,
			ylabel=$\log_{10}(\mathbb{P}(R > x))$,
			ymin=-6,
			ymax=1,
			xmin=1,
			xmax=2*10^6,
			no markers,
			xmode=log,
			legend pos= outer north east,
			legend cell align=left]
		\addlegendimage{empty legend}
		\addplot+ table [x=x,y=d1, col sep=comma]{\dataFigure};
		\addplot+ table [x=x,y=d2, col sep=comma]{\dataFigure};
		\addplot+ table [x=x,y=d3, col sep=comma]{\dataFigure};
		\addplot+[dashed] table [x=x,y=index, col sep=comma]{\dataFigure};
		\addlegendentry{Replicas:}
		\addlegendentry{$d=1$}
		\addlegendentry{$d=2$}
		\addlegendentry{$d=N=3$}
		\end{axis}
\end{tikzpicture}

%% file: TikZFigures/latencyTail_coc_identical_n_10_9_lambda_05_d.tex
\newcommand{\dataFigure}{TikZFigures/latencyTail_coc_identical_n_10_9_lambda_05_d.csv}

\begin{tikzpicture}
		\begin{axis}[
			xlabel=$x$,
			ylabel=$\log_{10}(\mathbb{P}(R > x))$,
			ymin=-6,
			ymax=0,
			xmin=1,
			xmax=2*10^6,
			no markers,
			xmode=log,
			legend pos= outer north east,
			legend cell align=left]
		\addlegendimage{empty legend}
		\addplot+ table [x=x,y=d1, col sep=comma]{\dataFigure};
		\addplot+ table [x=x,y=d2, col sep=comma]{\dataFigure};
		\addplot+ table [x=x,y=d3, col sep=comma]{\dataFigure};
		\addplot+[dashed] table [x=x,y=index_lower, col sep=comma]{\dataFigure};
		\addplot+[dashed] table [x=x,y=index_upper, col sep=comma]{\dataFigure};
		\addlegendentry{Replicas:}
		\addlegendentry{$d=1$}
		\addlegendentry{$d=2$}
		\addlegendentry{$d=N=3$}
		\end{axis}
\end{tikzpicture}

%% file: TikZFigures/latencyTail_coc_independent_n_10_9_lambda_05_d.tex
\newcommand{\dataFigure}{TikZFigures/latencyTail_coc_independent_n_10_9_lambda_05_d.csv}

\begin{tikzpicture}
		\begin{axis}[
			xlabel=$x$,
			ylabel=$\log_{10}(\mathbb{P}(R > x))$,
			ymin=-6,
			ymax=0,
			xmin=1,
			xmax=2*10^6,
			no markers,
			xmode=log,
			legend pos= outer north east,
			legend cell align=left]
		\addlegendimage{empty legend}
		\addplot+ table [x=x,y=d1, col sep=comma]{\dataFigure};
		\addplot+ table [x=x,y=d2, col sep=comma]{\dataFigure};
		\addplot+ table [x=x,y=d3, col sep=comma]{\dataFigure};
		\addplot+[dashed] table [x=x,y=index_lower, col sep=comma]{\dataFigure};
		\addplot+[dashed] table [x=x,y=index_upper, col sep=comma]{\dataFigure};
		\addlegendentry{Replicas:}
		\addlegendentry{$d=1$}
		\addlegendentry{$d=2$}
		\addlegendentry{$d=N=3$}
		\end{axis}
\end{tikzpicture}

%% file: TikZFigures/latencyTail_LCFS_coc_identical_n_10_9_lambda_05_d.tex
\newcommand{\dataFigure}{TikZFigures/latencyTail_LCFS_coc_identical_n_10_9_lambda_05_d.csv}

\begin{tikzpicture}
		\begin{axis}[
			xlabel=$x$,
			ylabel=$\log_{10}(\mathbb{P}(R > x))$,
			ymin=-10,
			ymax=0,
			xmin=1,
			xmax=2*10^6,
			no markers,
			xmode=log,
			legend pos= outer north east,
			legend cell align=left]
		\addlegendimage{empty legend}
		\addplot+ table [x=x,y=d1, col sep=comma]{\dataFigure};
		\addplot+ table [x=x,y=d2, col sep=comma]{\dataFigure};
		\addplot+ table [x=x,y=d3, col sep=comma]{\dataFigure};
		\addplot+[dashed] table [x=x,y=index_lower, col sep=comma]{\dataFigure};
		\addplot+[dashed] table [x=x,y=index_upper, col sep=comma]{\dataFigure};
		\addlegendentry{Replicas:}
		\addlegendentry{$d=1$}
		\addlegendentry{$d=2$}
		\addlegendentry{$d=N=3$}
		\end{axis}
\end{tikzpicture}

%% file: Redundancy_RV_arXiv.bbl
\begin{thebibliography}{10}

\bibitem{A-SEASP}
S.~Asmussen.
\newblock Subexponential asymptotics for stochastic processes: Extremal
  behaviour, stationary distributions and first passage times.
\newblock {\em Annals of Applied Probability}, 8:354--374, 1997.

\bibitem{ABV-RCOSJSW}
U.~Ayesta, T.~Bodas, and I.M. Verloop.
\newblock On redundancy-$d$ with cancel-on-start a.k.a.\
  join-shortest-work($d$).
\newblock {\em ACM SIGMETRICS Performance Evaluation Review}, 46(2):24--26,
  2019.

\bibitem{BMS-FJQ}
F.~Baccelli, A.M. Makowski, and A.~Shwartz.
\newblock The fork-join queue and related systems with synchronization
  constraints: {S}tochastic ordering and computable bounds.
\newblock {\em Advances in Applied Probability}, 21(3):629--660, 1989.

\bibitem{BGT-RV}
N.H. Bingham, C.M. Goldie, and J.L. Teugels.
\newblock {\em Regular {V}ariation.}
\newblock Cambridge University Press, 1987.

\bibitem{BM-TADHL}
J.~Blanchet and K.R.A. Murthy.
\newblock Tail asymptotics for delay in a half-loaded {GI/GI/2} queue with
  heavy-tailed job sizes.
\newblock {\em Queueing Systems}, 81:301--340, 2015.

\bibitem{BFLN-WTAROS}
O.J. Boxma, S.G. Foss, J.M. Lasgouttes, and
  R.~N$\acute{\text{u}}\tilde{\text{n}}$ez Queija.
\newblock Waiting time asymptotics in the single server queue with service in
  random order.
\newblock {\em Queueing Systems}, 46:35--73, 2004.

\bibitem{DB-TAS}
J.~Dean and L.Z. Barroso.
\newblock The tail at scale.
\newblock {\em Communications of the ACM}, 56(2):74--80, 2013.

\bibitem{DG-MRLC}
J.~Dean and S.~Ghemawat.
\newblock Map{R}educe: {S}implified data processing on large clusters.
\newblock {\em Communications of the ACM}, 51(1):107--113, 2008.

\bibitem{DG-DPT}
J.~Dean and S.~Ghemawat.
\newblock Map{R}educe: {A} flexible data processing tool.
\newblock {\em Communications of the ACM}, 53(1):72--77, 2010.

\bibitem{FH-TPQATD}
L.~Flatto and S.~Hahn.
\newblock Two parallel queues created by arrivals with two demands.
\newblock {\em SIAM Journal on Applied Mathematics}, 44(5):1041--1053, 1984.

\bibitem{FC-OSPAMS}
S.~Foss and N.~Chernova.
\newblock On the stability of a partially accessible multi-station queue with
  state-dependent routing.
\newblock {\em Queueing Systems}, 29:55--73, 1998.

\bibitem{FK-HTMS}
S.~Foss and D.~Korshunov.
\newblock Heavy tails in multi-server queues.
\newblock {\em Queueing Systems}, 52(1):31--48, 2006.

\bibitem{FK-LDMSHV}
S.~Foss and D.~Korshunov.
\newblock On large delays in multi-server queues with heavy tails.
\newblock {\em Mathematics of Operations Research}, 37(2):201--218, 2012.

\bibitem{FKZ-IHTSED}
S.~Foss, D.~Korshunov, and S.~Zachary.
\newblock {\em An Introduction to Heavy-Tailed and Subexponential
  Distributions.}
\newblock Springer, New York, 2013.

\bibitem{FR-SSRV}
S.~Foss and A.~Richards.
\newblock On sums of conditionally independent subexponential random variables.
\newblock {\em Mathematics of Operations Research}, 35(1):102--119, 2010.

\bibitem{GZDHBHSW-RLR}
K.S. Gardner, S.~Zbarsky, S.~Doroudi, M.~Harchol-Balter, E.~Hyytia, and
  A.~Scheller-Wolf.
\newblock Reducing latency via redundant requests: Exact analysis.
\newblock {\em ACM SIGMETRICS Performance Evaluation Review}, 43(1):347--360,
  2015.

\bibitem{H-OPQT}
M.~Harchol-Balter.
\newblock Open problems in queueing theory inspired by datacenter computing.
\newblock {\em Queueing Systems}, 97:3--37, 2021.

\bibitem{HRS-PBOCQ}
D.~Heath, S.~Resnick, and G.~Samorodnitsky.
\newblock Patterns of buffer overflow in a class of queues with long memory in
  the input stream.
\newblock {\em Annals of Applied Probability}, 7:1021--1057, 1997.

\bibitem{HH-PRR}
T.~Hellemans and B.~Van Houdt.
\newblock Performance of redundancy-$d$ with identical/independent replicas.
\newblock {\em ACM Transactions on Modeling and Performance Evaluation of
  Computing Systems}, 4(2):1--28, 2019.

\bibitem{JM-RVF}
A.H. Jessen and T.~Mikosch.
\newblock Regularly varying functions.
\newblock {\em Publications de l'Institut Math$\acute{\text{e}}$matique},
  80(94):171--192, 2006.

\bibitem{M-RVSAP}
T.~Mikosch.
\newblock Regular variation, subexponentiality and their applications in
  probability theory.
\newblock Lecture notes, University of Groningen, 1999.
\newblock \url{https://www.eurandom.tue.nl/pre-prints/} 1999-013.

\bibitem{NT-AAFJ}
R.D. Nelson and A.N. Tantawi.
\newblock Approximate analysis of fork/join synchronization in parallel queues.
\newblock {\em IEEE Transactions on Computers}, 37(6):739--743, 1988.

\bibitem{PW-SSNTPE}
K.~Park and W.~Willinger.
\newblock {\em Self-{S}imilar {N}etwork {T}raffic and {P}erformance
  {E}valuation.}
\newblock John Wiley and Sons, Inc., New York, 2000.

\bibitem{RBB-DPP}
Y.~Raaijmakers, S.C. Borst, and O.J. Boxma.
\newblock Delta probing policies for redundancy.
\newblock {\em Performance Evaluation}, 127-128:21--35, 2018.

\bibitem{RBB-RSSB}
Y.~Raaijmakers, S.C. Borst, and O.J. Boxma.
\newblock Redundancy scheduling with scaled {B}ernoulli service requirements.
\newblock {\em Queueing Systems}, 93(1-2):67--82, 2019.

\bibitem{RBB-STBPS}
Y.~Raaijmakers, S.C. Borst, and O.J. Boxma.
\newblock Stability and tail behavior of redundancy systems with processor
  sharing.
\newblock {\em Performance Evaluation, to appear}, 147:1--40, 2021.

\bibitem{R-HTA}
S.~Resnick.
\newblock Heavy tailed analysis.
\newblock Lecture notes, Cornell University, 2005.
\newblock \url{https://www.eurandom.tue.nl/pre-prints/} 2005-024.

\bibitem{T-AFJ}
A.~Thomasian.
\newblock Analysis of fork/join and related queueing systems.
\newblock {\em ACM Computing Surveys}, 47(2):1--71, 2014.

\bibitem{V-ABWHF}
N.~Veraverbeke.
\newblock Asymptotic behaviour of {W}iener-{H}opf factors of a random walk.
\newblock {\em Stochastic Processes and their Applications}, 5(1):27--37, 1977.

\bibitem{VGMSRS-LR}
A.~Vulimiria, P.B. Godfrey, R.~Mittal, J.~Sherry, S.~Ratnasamy, and S.~Shenker.
\newblock Low latency via redundancy.
\newblock {\em CoNEXT '13: Proceedings of the ninth ACM Conference on Emerging
  Networking Experiments and Technologies}, pages 283--294, 2013.

\bibitem{Z-QSHT}
A.P. Zwart.
\newblock {\em Queueing {S}ystems with {H}eavy {T}ails.}
\newblock PhD thesis, Eindhoven University of Technology, 2001.
\newblock
  \url{https://research.tue.nl/nl/publications/queueing-systems-with-heavy-tails}.

\bibitem{Z-TABP}
A.P. Zwart.
\newblock Tail asymptotics for the busy period in the ${GI}/{G}/1$ queue.
\newblock {\em Mathematics of Operations Research}, 26(3):485--493, 2001.

\bibitem{ZB-SAPS}
A.P. Zwart and O.J. Boxma.
\newblock Sojourn time asymptotics in the ${M}/{G}/1$ processor sharing queue.
\newblock {\em Queueing Systems}, 35:141--166, 2000.

\end{thebibliography}
